\documentclass[12pt,leqeq]{article}
\usepackage{amssymb,amsthm}
\usepackage{amsmath}
\usepackage{amscd}
\usepackage{graphicx}
\usepackage{subfigure}
\usepackage{epsf,graphicx}
\usepackage{epstopdf}
\newtheorem{thm}{Theorem}
\newtheorem{lema}[thm]{Lemma}
\newtheorem{prop}[thm]{Proposition} 

\newtheorem{defi}[thm]{Definition}
\newtheorem{exa}[thm]{Example}

\begin{document}
\date{}

\title{On  the Combinatorics  of  the Universal Enveloping Algebra $\ \widehat{U}_h(\mathfrak{sl}_2)$ }
\author{Rafael D\' iaz \ \ and \ \ Edward Salamanca}
\maketitle

\begin{abstract}
We study using combinatorial methods the structural coefficients of the formal homogeneous universal enveloping algebra $\ \widehat{U}_h(\mathfrak{sl}_2)\ $  of   the special linear algebra $\ \mathfrak{sl}_2,\ $ over a field of characteristic zero. We provide explicit formulae for the product of generic elements in $\ \widehat{U}_h(\mathfrak{sl}_2),\ $ and construct combinatorial objects  giving flesh to  these formulae.

\end{abstract}

\section{Introduction and Basic Notions}

When studying an algebra $\ A \ $ we usually approach it from a couple of complementary viewpoints. On the one hand we study $\ A \ $ by given generators for it, generators for the relations among generators, generators for the relations among relations, and so on, i.e. we study $\ A \ $ by analyzing a free resolution for it. This approach was pioneered by Hilbert. On the other hand we study $\ A \ $ by finding a  basis $\ \{e_i\} \ $ for $\ A, \ $ and studying the structural coefficients $\ c_{ij}^k\ $ defined by the equation $\ \ \displaystyle e_i e_j \ = \ \sum_{k}c_{ij}^k e_k. \ \ $ Note that we do not demand that $\ A \ $ be finite dimensional, but do insist that the sum above be finite; otherwise further topological considerations have to be imposed. For example, one says that $A$ is a combinatorial (resp. integral, rational) algebra if it admits a combinatorial (resp. integral, rational) basis, i.e. a basis such that the associated structural coefficients are such that $\  c_{ij}^k \in \mathbb{N} \ $ (resp. $c_{ij}^k \in  \mathbb{Z},  \   c_{ij}^k \in \mathbb{Q}.) $ Analogously, for  $\ k = \mathbb{R}, \ $ we say that $\ A \ $ is a probabilistic algebra, if it admits a basis such that the associated structural coefficients satisfy $ \ \ c_{ij}^k \geq 0 \ \  \mbox{and}  \ \ \sum_{k}c_{ij}^k = 1. \ \ $ This approach, strongly promoted by Rota \cite{ro},  has been undertaken by many combinatorialist, for example, in the study of combinatorial Hopf algebras \cite{nb, gri, lo, sch}.\\

We assume the reader to be familiar with the rudiments of the theory of Lie algebras \cite{bak, re3, dixm, erd, fult, hum, jac, kass}.  Fix a field $\ k \ $ of characteristic zero.  Given a Lie algebra  $\ L \ $  over $\ k, \ $ we let $\ T(L)\ $ be the free $\ k$-algebra generated by $\ L, \ $ and  $\ U(L)\ = \ T(L)/I(L) \ $ be the universal enveloping algebra of $\ L, \ $ where  $\ I(L)\ $ is the two-sided ideal of  $\ T(L)\ $ generated by elements of the form $\ \ xy\ -\ yx\ -\ [x,y] \ \  \mbox{for}  \ \ x,y \in L.$\\

Consider the Lie algebra $\  \mathrm{M}_{2\times 2} (k)\ $ of square matrices of size two with entries in $k$, and its
Lie sub-algebra $\ \mathfrak{sl}_2  =    \lbrace A \in \mathrm{M}_{2\times 2} (k) \ \vert \ \mathrm{tr} (A)=0  \rbrace \ $ consisting of trace zero matrices. The universal enveloping algebra $\ U(\mathfrak{sl}_2)\ $ is given by $$U(\mathfrak{sl}_2)\ \ = \ \ T(\mathfrak{sl}_2)/I(\mathfrak{sl}_2)\  \ = \ \ \frac{k<x,y,z>}{I},$$ i.e.
$\ U(\mathfrak{sl}_2)\ $ is the quotient of the free algebra $\ k<x,y,z> \ $ generated by $\ x, \ y \ $ and $\ z $ by the ideal $\ I\ $ generated by the identities:
$$yx \  =  \ xy+2x, \ \ \ \ \ \ \ \ zx \  =  \ xz-y, \ \ \ \ \ \ \ \ zy \  =  \ yz+2z.$$

Pick a variable $\ h \ $ algebraically independent from $\ x,y,z. \ $ The homogeneous universal enveloping algebra $\ U_h(\mathfrak{sl}_2)\ $ is given by $$U_h(\mathfrak{sl}_2)\ \ \ = \ \ \ \frac{k<x,y,z>[h]}{I_h} ,$$ i.e. $\ U_h(\mathfrak{sl}_2)\ $ is the free associative $\ k$-algebra generated by the letters $\ x, y, z\ $ and $\ h, \ $ divided by the relations:
$$yx\  =  \ xy+2xh, \ \ \ \ \ \ \ \ zx  \ =  \ xz-yh, \ \ \ \ \ \ \ \ zy \ = \ yz+2zh,$$
and $\ h \  $ commutes with $\ x,  y,    z. \ $ The methods employed in this work apply for homogeneous universal enveloping algebra $\ U_h(\mathfrak{sl}_2), \ $  nevertheless we
are going to develop our results for the formal homogeneous universal enveloping algebra $\ \widehat{U}_h(\mathfrak{sl}_2)\ $  given by
$$\widehat{U}_h(\mathfrak{sl}_2)\ \ \ = \ \ \ \frac{k<<x,y,z>>[[h]]}{I_h} ,$$ i.e. we allow formal power series in the variables $\ x, y, z, h \ $ subject to the same relations generating the ideal $\ I_h \ $ from above.\\

Fix the order $\ x < y < z < h\ $ on the formal generators of $\ \widehat{U}_h(\mathfrak{sl_2}). \ $ A monomial in $\ x, y,z,  h \ $ is in normal form if it looks like $\ x^ay^bz^ch^d.\ $ By the Poincar\'e-Birkhoff-Witt theorem \cite{dixm, hum} all elements of $\ \widehat{U}_h(\mathfrak{sl_2}) \ $ can be written in a unique way as a formal sum of monomials in normal form, i.e. a generic element of $\ \widehat{U}_h(\mathfrak{sl_2}) \ $ can be written as
$$f \ \  =  \ \ \sum_{a,b,c,d \in \mathbb{N}} f_{a,b,c,d}\frac{x^ay^bz^ch^d}{a!b!c!d!}.$$
Note that we use divided power monomials, i.e. monomials divided by the corresponding factorials.  This is the main reason why we work with a characteristic zero field. The basis of divided power monomials is a priory a rational basis; a main result of this work is to
show that it is actually an integral basis, i.e. we show that the product of formal power series in normal form with divided power monomials can be written in normal form with divided power monomials  through a sophisticated procedure, which we are going to dissect first from an algebraic viewpoint, and then adopting a combinatorial viewpoint. The reader may take a look at Figure \ref{f4} which, as we shall argue, summarizes the combinatorics behind $\ \widehat{U}_h(\mathfrak{sl_2}) \ $ arising from the basis of divided power monomials written in the prescribed order.  Working with divided powers makes our combinatorial constructions more transparent, nevertheless we stress that they can be adapted to work with undivided monomials as well.

\section{Explicit Formulae for the Product on $\ \widehat{U}_{h}(\mathfrak{sl}_2)$}

In this section we provide explicit formulae for the product on $\ \widehat{U}_{h}(\mathfrak{sl}_2). \ $ Before attacking our problem we first introduced the required notation. \\

For $n \in \mathbb{N}$ we set $\ [n]=\{1,...,n \}\ $ for $\ n\geq 1,\ $
and $\ [0]=\emptyset.\ $ If $\ x \ $ is a finite set we let $\ |x| \ $ be its cardinality and set $\ [z]=[|z|].$\\

For $\ s,n \in \mathbb{N}\ $ with $\ 0\leq s\leq n, \ $ we let the elementary symmetric function $\ e^{s}_n\ $ on the variables $\ x_1,...,x_n\ $ be given by $$e^{s}_{n}(x_1,...,x_n)\ \ =  \  \ \displaystyle \sum_{1\leq i_1 <...<i_s \leq n}^{} \;x_{i_{1}}...x_{i_{s}}\ \  =   \ \ \underset{\underset{|A|=s}{A\subseteq [n]}}{\sum} \  \prod_{i\in A}\:x_i .$$
For $\ a \in \mathbb{N} \ $ and $\ 0\leq s\leq n, \ $  we set  $\ \ \displaystyle (a)^s_n\ = \ e^{s}_{n}(a,a+1,...,a+n-1).\ $
For $\ 1\leq s\leq n, \ $ the symbol $\ (a)^s_n \ $ satisfies the recursion:
$$(a)^s_n \ \ = \ \ (a)^{s}_{n-1} \ + \ (a+n-1)(a)^{s-1}_{n-1},$$
as well as  the boundary conditions
$$(a)^s_n \ = \ 0, \ \ \ \mbox{for}\ \ \ s > n,\ \ \ \ \ \ \mbox{and}\ \ \ \ \ \ (a)^0_n \ = \ 1, \ \ \ \mbox{for} \ \ \ n \geq 0. $$
For example, for $\ n=4\ $ and $\ s=3,\ $ we have that
$$\ \ (a)^{3}_{4}\ \ = \ \ e^{3}_{4}(a,a+1,a+2,a+3) \ \  =  \ \ 4a^{3}\ + \ 18a^{2}\ + \ 22a\ +\ 6.$$
\begin{figure}[t]
  \centering
  \includegraphics[width=4cm]{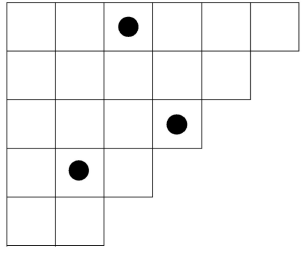}\label{f1}
  \caption{Configuration contributing to $\ (2)^3_5.$}
\end{figure}A combinatorial interpretation for $\ (a)^s_n \ $ in terms of tableaux goes as follows: $\ (a)^s_n \ $ counts the number of
ways of inserting $\ s \ $ unlabeled dots into a tableaux with rows of length $\ (a +n-1,a+n-2,...,a+1,a),\ $ with at most one dot in each row.  Figure \ref{f1} shows an example of a configuration contributing to the computation of $\ (2)^3_5 .\ $\\

In the proof of  Proposition \ref{bd} below we use the Pochhammer
$\ k$-symbol \cite{re4, mu} which is given for $\ a,k \ $ in some algebra $\ A\ $ and  $\ n \in \mathbb{N}_{>0}\ $  by
$$(a)_{n,k}\  =  \ a(a+k)(a+2k)\ .\ .\ . \ (a+(n-1)k) .$$
We let $\ (a)_{n} =  a(a-1)(a-2)\ .\ . \ . \ (a-n+1) \ $ be the falling factorial for $ \ a,n \in \mathbb{N}. \ $
Note that $\ (a)_n = (a)_{n,-1}.$ \\

We are also going to use the following identities.

\begin{lema}\label{vam}
{\em For $\ a,h \ $ in an algebra $\ A, \ $ and $\ n \in \mathbb{N}_{>0}\ $ we have that:
\begin{enumerate}
  \item $ (a)_{n,-h}\  =  \ \big(a-(n-1)h\big)(a)_{n-1,-h} .$
  \item $ (a)_{n,-h}\  =  \  a\big(a-h\big)_{n-1,-h}. $
 \item $\displaystyle \bigg(y+\big(a-(n+1)\big)h\bigg)_{n,-h}\ =  \ \sum_{w=0}^{n}\big(a-2n\big)_{n}^{w}y^{n-w}h^{w}.$
\end{enumerate}
}
\end{lema}

\begin{proof}Properties 1 and 2 follow from  definition. To show  3 note that
$$\bigg(y+\big(a-(n+1)\big)h\bigg)_{n,-h} \  =  \  \prod_{i=0}^{n-1}\bigg[y \ + \ \bigg( a \ - \ \big(n+i+1 \big)  \bigg)h \bigg].$$
Developing the latter product  assume we choose  $\ w \in [n] \ $ times a $h$-term   and thus $\ n-w \ $ times a $\ y$-term, this choice gives rise just to the sum $\ \sum_{w=0}^{n}\ $ and the factor $\ y^{n-w}h^{w}\ $ in our proposed formula. In order to get the factor $\ h^w\ $ we must have chosen and multiplied $\ w \ $ elements from the set
$$\big(a - (n+1)\big), \  \ \big(a - (n+2)\big), \ \dots ,  \ \  \big(a - (2n)\big), $$ which agrees, by definition, with the result of computing
$\ (a-2n\big)_{n}^{w}.$
\end{proof}

Next proposition forms the basis upon which all further results in this work are built. Related identities are given by  Kac  \cite{kac},
and by D\'iaz and Pariguan \cite{re}.

\begin{prop}\label{bd}
{\em
For $\ a,b \in \mathbb{N}, \ $ the following identities hold in $\ \widehat{U}_h(sl_2)$:
\begin{enumerate}
\item{$\displaystyle \frac{z^{a}}{a!}\frac{y^{b}}{b!}\ \ = \ \ \displaystyle \sum_{k=0}^{b} \; (2a)^{k}\frac{y^{b-k}}{(b-k)!}\frac{z^{a}}{a!}\frac{h^{k}}{k!}.$}
\item{$\displaystyle  \frac{y^{a}}{a!}\frac{x^{b}}{b!}\ \ = \ \ \displaystyle \sum_{k=0}^{a} \; {(2b)^k }\frac{x^{b}}{b!}\frac{y^{a-k}}{(a-k)!}\frac{h^{k}}{k!}.$}
\item{$\displaystyle  \frac{z^{a}}{a!}\frac{x^{b}}{b!}\ \ = $
 $$ \sum_{0\leq w \leq v \leq \mathrm{min}(a,b)} \;(-1)^{v}(v-w)!(v+w)_{w}(a+b-2v)^{w}_{v}\frac{x^{b-v}}{(b-v)!} \frac{y^{v-w}}{(v-w)!}\frac{z^{a-v}}{(a-v)!}\frac{h^{v+w}}{(v+w)!}.$$}
\end{enumerate}
}
\end{prop}

\begin{proof}We show 1 and 3, item  2 follows from 1 using the automorphism  of $\ \widehat{U}_h(\mathfrak{sl}_2) \ $ given on generators by
$$x \ \rightarrow  \ z, \ \ \ \ \ \ \  y \ \rightarrow \ -y, \ \ \ \ \ \ \ z \ \rightarrow \ x, \ \ \ \ \ \ \  \mbox{and} \ \ \ \ \ \ \  h \ \rightarrow \ h.$$
We show item 1 by induction. First we check that $$\ \frac{z^a}{a!}y \ \ = \ \ y\frac{z^{a}}{a!}\ + \ 2a\frac{z^{a}}{a!}h. \ $$ For $\ a=1\ $ we get a defining identity for $\ \widehat{U}_h(\mathfrak{sl}_2). \ $ For
$\ a>0 \ $ we have that:
$$\frac{z^{a+1}}{(a+1)!}y \ \ \ = \ \ \ \frac{z}{(a+1)}\left( \frac{z^{a}}{a!}y \right)
\ \ \ = \ \ \ \frac{z}{(a+1)}\left( y \frac{z^{a}}{a!}\ + \ 2a \frac{z^{a}}{a!}h \right) \ \ \ = $$
$$\frac{1}{(a+1)}\left( yz\ + \ 2zh \right)\frac{z^{a}}{a!}\ \ + \ \ 2a\frac{z^{a+1}}{(a+1)!}h \ \ \ = \ \ \ y\frac{z^{a+1}}{(a+1)!}\ \ \ + \ \ \ 2(a+1)\frac{z^{a+1}}{(a+1)!}h.$$
We proceed by induction on $\ b. \ $ We have that:
$$\frac{z^{a}}{a!}\frac{y^{b+1}}{(b+1)!} \ \ \ = \ \ \
\left( \frac{z^{a}}{a!}\frac{y^{b}}{b!} \right)\frac{y}{(b+1)} \ \ \ = \ \ \
\sum_{k=0}^{b} \;\frac{(2a)^{k}}{(b+1)}\frac{y^{b-k}}{(b-k)!}\left( \frac{z^{a}}{a!}y \right) \frac{h^{k}}{k!} \ \ \ = $$
$$\sum_{k=0}^{b} \;\frac{(2a)^{k}}{(b+1)}\frac{y^{b-k}}{(b-k)!}\left( y \frac{z^{a}}{a!} \ \ + \ \ (2a)\frac{z^{a}}{a!}h \right) \frac{h^{k}}{k!} \ \ \ = $$
$$\sum_{k=0}^{b} \;\frac{(2a)^{k}}{(b+1)}\frac{y^{b+1-k}}{(b-k)!}\frac{z^{a}}{a!}\frac{h^{k}}{k!} \ \  +  \ \  \sum_{k=0}^{b} \frac{(2a)^{k+1}}{(b+1)}\frac{y^{b-k}}{(b-k)!}\frac{z^{a}}{a!}\frac{h^{k+1}}{k!}\ \ \ = $$
$$\frac{y^{b+1}}{(b+1)!}\frac{z^{a}}{a!}\ \ \ + \ \ \ \sum_{k=1}^{b} \;\frac{(2a)^{k}}{(b+1)}\frac{y^{b+1-k}}{(b-k)!}\frac{z^{a}}{a!}\frac{h^{k}}{k!}
\ \ \ + $$
$$\sum_{k=0}^{b-1}  \frac{(2a)^{k+1}}{(b+1)}\frac{y^{b-k}}{(b-k)!}\frac{z^{a}}{a!}\frac{h^{k+1}}{k!}
\ \ \ + \ \ \ (2a)^{b+1}\frac{z^{a}}{a!}\frac{h^{b+1}}{(b+1)!} \ \ \ = $$
$$\frac{y^{b+1}}{(b+1)!}\frac{z^{a}}{a!}\ \ \ + \ \ \ \sum_{k=1}^{b} \;(2a)^{k}  \frac{y^{b+1-k}}{(b+1-k)!}\frac{z^{a}}{a!}\frac{h^{k}}{k!}\ \ \  + \ \ \ (2a)^{b+1}\frac{z^{a}}{a!}\frac{h^{b+1}}{(b+1)!} \ \ \ = $$
$$\sum_{k=0}^{b+1} \;(2a)^{k}\frac{y^{b+1-k}}{(b+1-k)!}\frac{z^{a}}{a!}\frac{h^{k}}{k!}.$$
Next we show item 3. From Lemma \ref{vam}  we have that
$$\bigg(y+\big(a+b-(v+1)\big)h\bigg)_{v,-h}\ \ = \ \ \sum_{w=0}^{v} \big(a+b-2v\big)_{v}^{w}y^{v-w}h^{w}.$$
Thus our desired identity is equivalent to
$$\frac{z^{a}}{a!}\frac{x^{b}}{b!}\ \ \ = \ \ \ \sum_{v=0}^{\mathrm{min}(a,b)} \;(-1)^{v}\frac{x^{b-v}}{(b-v)!}\bigg(y+\big(a+b-(v+1)\big)h\bigg)_{v,-h}\frac{z^{a-v}}{(a-v)!}\frac{h^{v}}{v!}.$$
By induction on $\ b \ $ one shows that:
$$z \frac{x^{b}}{b!} \ \  =  \ \ \frac{x^{b}}{b!}z \ \ - \ \ \frac{x^{b-1}}{(b-1)!}\big( y+(b-1)h \big)h.$$
We proceed from this formula by induction on $\ a$.
\begin{footnotesize}
$$	 \frac{z^{a+1}}{(a+1)!}\frac{x^{b}}{b!} \ \ \ =  \ \ \ \frac{z}{(a+1)} \left(\frac{z^{a}}{a!} \frac{x^{b}}{b!} \right)\ \ \ = $$
$$  \sum_{v=0}^{\mathrm{min}(a,b)} \frac{(-1)^{v}}{(a+1)} \left( z\frac{x^{b-v}}{(b-v)!} \right) \bigg( y+\big(a+b-v-1\big)h \bigg)_{v,-h} \frac{z^{a-v}}{(a-v)!}\frac{h^{v}}{v!} \ \ \ = $$
$$  \sum_{v=0}^{\mathrm{min}(a,b)} \frac{(-1)^{v}}{(a+1)} \left( \frac{x^{b-v}}{(b-v)!}z\ - \ \frac{x^{b-v-1}}{(b-v-1)!} \big( y+(b-v-1)h \big)h \right)
\bigg( y+\big(a+b-v-1\big)h \bigg)_{v,-h} \frac{z^{a-v}}{(a-v)!}\frac{h^{v}}{v!} \ \ \ =$$
$$\sum_{v=0}^{\mathrm{min}(a,b)}\frac{(-1)^{v}}{(a+1)} \frac{x^{b-v}}{(b-v)!}  z\bigg( y+\big(a+b-v-1\big)h \bigg)_{v,-h}  \frac{z^{a-v}}{(a-v)!}\frac{h^{v}}{v!} \ \ \ + $$
$$\sum_{v=0}^{\mathrm{min}(a,b)}\frac{(-1)^{v+1}}{(a+1)} \big(v+1\big)\frac{x^{b-v-1}}{(b-v-1)!} \bigg( y+\big(b-v-1\big)h \bigg)
\bigg( y+\big(a+b-v-1\big)h \bigg)_{v,-h}\frac{z^{a-v}}{(a-v)!}\frac{h^{v+1}}{(v+1)!}\ \ \ = $$
$$ \sum_{v=0}^{\mathrm{min}(a,b)}\frac{(-1)^{v}}{(a+1)} \frac{x^{b-v}}{(b-v)!}  \bigg( y+\big(a+b-v+1\big)h \bigg)_{v,-h} z\frac{z^{a-v}}{(a-v)!}\frac{h^{v}}{v!}\ \ \ + $$
$$\sum_{v=0}^{\mathrm{min}(a,b)}\frac{(-1)^{v+1}}{(a+1)} \big(v+1\big)\frac{x^{b-v-1}}{(b-v-1)!} \bigg( y+\big(b-v-1\big)h \bigg)
\bigg( y+\big(a+b-v-1\big)h \bigg)_{v,-h}\frac{z^{a-v}}{(a-v)!}\frac{h^{v+1}}{(v+1)!} \ \ =$$
$$ \sum_{v=0}^{\mathrm{min}(a,b)}\frac{(-1)^{v}}{(a+1)}\big(a+1-v\big) \frac{x^{b-v}}{(b-v)!} \bigg( y+\big(a+b-v+1\big)h \bigg)_{v,-h} \frac{z^{a+1-v}}{(a+1-v)!}\frac{h^{v}}{v!}\ \ \ +$$
$$\sum_{v=0}^{\mathrm{min}(a,b)}\frac{(-1)^{v+1}}{(a+1)} \big(v+1\big)\frac{x^{b-v-1}}{(b-v-1)!} \bigg( y+\big(b-v-1\big)h \bigg)
\bigg( y+\big(a+b-v-1\big)h \bigg)_{v,-h}\frac{z^{a-v}}{(a-v)!}\frac{h^{v+1}}{(v+1)!}.$$
\end{footnotesize}
Making the change $\ v \rightarrow v+1 \ $ we get that $\ \frac{z^{a+1}}{(a+1)!}\frac{x^{b}}{b!}\ $ is equal to:
\begin{footnotesize}
$$ \sum_{v=0}^{\mathrm{min}(a,b)}\frac{(-1)^{v}}{(a+1)}\big(a+1-v\big) \frac{x^{b-v}}{(b-v)!}  \bigg( y+\big(a+b-v+1\big)h \bigg)_{v,-h} \frac{z^{a+1-v}}{(a+1-v)!}\frac{h^{v}}{v!}\ \ \ +$$
$$\sum_{v=1}^{\mathrm{min}(a+1,b)}\frac{(-1)^{v}}{(a+1)} v\frac{x^{b-v}}{(b-v)!} \bigg( y+\big(b-v\big)h \bigg)
\bigg( y+\big(a+b-v\big)h \bigg)_{v-1,-h}\frac{z^{a+1-v}}{(a+1-v)!}\frac{h^{v}}{v!} \ \ \ = $$
$$\frac{x^{b}}{b!}\frac{z^{a+1}}{(a+1)!} \ \ \ + $$
$$\sum_{v=1}^{\mathrm{min}(a,b)}\frac{(-1)^{v}}{(a+1)}\big(a+1-v\big) \frac{x^{b-v}}{(b-v)!} \bigg( y+\big(a+b-v+1\big)h \bigg)_{v,-h} \frac{z^{a+1-v}}{(a+1-v)!}\frac{h^{v}}{v!} \ \ \ +$$
$$\sum_{v=1}^{\mathrm{min}(a+1,b)}\frac{(-1)^{v}}{(a+1)} v\frac{x^{b-v}}{(b-v)!} \bigg( y+\big(b-v\big)h \bigg)
\bigg( y+\big(a+b-v\big)h \bigg)_{v-1,-h}\frac{z^{a+1-v}}{(a+1-v)!}\frac{h^{v}}{v!} \ \ \ = $$
$$ \frac{x^{b}}{b!}\frac{z^{a+1}}{(a+1)!} \ \ + \ \ (-1)^{a+1}\frac{x^{b-a-1}}{(b-a-1)!}\bigg( y+\big(b-a-1\big)h \bigg)
\bigg( y+\big(b-1\big)h \bigg)_{a,-h}\frac{h^{a+1}}{(a+1)!}\ \ \ +$$
$$\sum_{v=1}^{\mathrm{min}(a,b)}\frac{(-1)^{v}}{(a+1)}\big(a+1-v\big) \frac{x^{b-v}}{(b-v)!} \bigg( y+\big(a+b-v+1\big)h \bigg)_{v,-h} \frac{z^{a+1-v}}{(a+1-v)!}\frac{h^{v}}{v!} \ \ \ + $$
$$\sum_{v=1}^{\mathrm{min}(a,b)}\frac{(-1)^{v}}{(a+1)} v\frac{x^{b-v}}{(b-v)!} \bigg( y+\big(b-v\big)h \bigg) \bigg( y+\big(a+b-v\big)h \bigg)_{v-1,-h}\frac{z^{a+1-v}}{(a+1-v)!}\frac{h^{v}}{v!}.$$
\end{footnotesize}
We apply to the latter expression the following identities coming from Lemma \ref{vam}  $$ \bigg(y+(b-1)h \bigg)_{a+1,-h}\ \ = \ \ \bigg(y+(b-a-1)h\bigg)\bigg(y+(b-1)h \bigg)_{a,-h}, $$
$$ \bigg(y+\big(a+b-v+1\big)h\bigg)_{v,-h}\  = \  \bigg(y+\big(a+b-v+1\big)h \bigg) \bigg(y+\big(a+b-v+1\big)h-h\bigg)_{v-1,-h}, $$
to get that
\begin{footnotesize}
$$\frac{x^{b}}{b!}\frac{z^{a+1}}{(a+1)!} \ \ + \ \ (-1)^{a+1}\frac{x^{b-a-1}}{(b-a-1)!} \bigg( y+\big(b-1\big)h \bigg)_{a+1,-h}\frac{h^{a+1}}{(a+1)!}\ \ \ +$$
$$\sum_{v=1}^{\mathrm{min}(a,b)}\frac{(-1)^{v}}{(a+1)} \frac{x^{b-v}}{(b-v)!}  (a+1-v)\bigg( y+\big(a+b-v+1\big)h \bigg)\bigg( y+\big(a+b-v\big)h \bigg)_{v-1,-h} \frac{z^{a+1-v}}{(a+1-v)!}\frac{h^{v}}{v!} \ \ + $$
$$\sum_{v=1}^{\mathrm{min}(a,b)}\frac{(-1)^{v}}{(a+1)} \frac{x^{b-v}}{(b-v)!} \left( v\left( y+(b-v)h \right)\left( y+(a+b-v)h \right)_{v-1,-h} \right) \frac{z^{a+1-v}}{(a+1-v)!}\frac{h^{v}}{v!}.$$
\end{footnotesize}
Taking $\ \ \frac{x^{b-v}}{(b-v)!}, \ \ \frac{z^{a+1-v}}{(a+1-v)!}, \ \ \frac{h^{v}}{v!} \ \ \mbox{and} \  \ \bigg(  y+\big(a+b-v\big)h \bigg)_{v-1,-h} \ \  $ as a common factors, the previous expression becomes:
\begin{footnotesize}
$$\frac{x^{b}}{b!}\frac{z^{a+1}}{(a+1)!} \ \ + \ \ (-1)^{a+1}\frac{x^{b-a-1}}{(b-a-1)!} \bigg( y+\big(b-1\big)h \bigg)_{a+1,-h}\frac{h^{a+1}}{(a+1)!}\ \ \ +$$
$$\sum_{v=1}^{\mathrm{min}(a,b)}\frac{(-1)^{v}}{(a+1)} \frac{x^{b-v}}{(b-v)!} \big(a+1\big) \bigg( y+\big(a+b-2v+1\big)h \bigg)
\bigg( y+\big(a+b-v\big)h \bigg)_{v-1,-h}  \frac{z^{a+1-v}}{(a+1-v)!}\frac{h^{v}}{v!} \ \ =$$
$$\frac{x^{b}}{b!}\frac{z^{a+1}}{(a+1)!} \ \ + \ \ (-1)^{a+1}\frac{x^{b-a-1}}{(b-a-1)!} \bigg( y+\big(b-1\big)h \bigg)_{a+1,-h}\frac{h^{a+1}}{(a+1)!}\ \ +$$
$$\sum_{v=1}^{\mathrm{min}(a,b)}(-1)^{v} \frac{x^{b-v}}{(b-v)!} \bigg( y+\big(a+b-2v+1\big)h \bigg)\bigg( y+\big(a+b-v\big)h \bigg)_{v-1,-h} \frac{z^{a+1-v}}{(a+1-v)!}\frac{h^{v}}{v!}. $$
\end{footnotesize}
We apply the following identity coming from Lemma \ref{vam}
$$ \bigg(y+\big(a+b-v\big)h \bigg)_{v,-h}\ \ = \ \ \bigg(y+\big(a+b-v\big)h\bigg)_{v-1,-h}\bigg(y+\big(a+b-2v-1\big)h\bigg) $$
to  the third term above to  get that:
$$\frac{x^{b}}{b!}\frac{z^{a+1}}{(a+1)!} \ \ + \ \ (-1)^{a+1}\frac{x^{b-a-1}}{(b-a-1)!} \bigg( y+\big(b-1\big)h \bigg)_{a+1,-h}\frac{h^{a+1}}{(a+1)!}\ \ \ + $$
$$\sum_{v=1}^{\mathrm{min}(a,b)}(-1)^{v} \frac{x^{b-v}}{(b-v)!} \bigg( y+\big(a+b-v\big)h \bigg)_{v,-h} \frac{z^{a+1-v}}{(a+1-v)!}\frac{h^{v}}{v!}\ \ \ =$$
$$\sum_{v=0}^{\mathrm{min}(a+1,b)}(-1)^{v} \frac{x^{b-v}}{(b-v)!} \bigg( y+\big(a+b-v\big)h \bigg)_{v,-h} \frac{z^{a+1-v}}{(a+1-v)!}\frac{h^{v}}{v!},$$
showing the desired result.

\end{proof}

\begin{exa}{\em For $\ a=1 \ $ and $\ b=2, \ $ the identities defining $\ \widehat{U}_{h}(\mathfrak{sl}_2) \ $ imply that:
$$z\frac{x^{2}}{2!} \ \ = \ \ (zx)\frac{x}{2!} \ \ = \ \ (xz-yh)\frac{x}{2!} \ \ = \ \ \frac{x}{2!}(zx)-(yx)\frac{h}{2!} \ \ = $$
$$\frac{x}{2!}(xz-yh)-(xy+2xh)\frac{h}{2!} \ \ = \ \ \frac{x^{2}}{2!}z-\frac{1}{2!}xyh-\frac{1}{2!}xyh-2x\frac{h^{2}}{2!} \ \ = $$
$$\ \ \frac{x^{2}}{2!}z\ - \ xyh\ - \ 2x\frac{h^{2}}{2!}.$$
On the other hand, using Proposition \ref{bd} we get that:
\begin{footnotesize}
$$z\frac{x^{2}}{2!} \ \ = \ \  \sum_{0 \leq w \leq v \leq 1}^{} \;(-1)^{v}(v-w)!(v+w)_{w}(3-2v)^{w}_{v}\frac{x^{2-v}}{(2-v)!}\frac{y^{v-w}}{(v-w)!}\frac{z^{1-v}}{(1-v)!}\frac{h^{v+w}}{(v+w)!}\ \ = $$
$$(-1)^{0}(0)!(0)_{0}(3)^{0}_{0}\frac{x^{2}}{2!}\frac{y^{0}}{0!}\frac{z^{1}}{1!}\frac{h^{0}}{0!}\ + \ (-1)^{1}(1)!(1)_{0}(1)^{0}_{1}\frac{x^{1}}{1!}\frac{y^{1}}{1!}\frac{z^{0}}{0!}
\frac{h^{1}}{1!}\ + \ (-1)^{1}(0)!(2)_{1}(1)^{1}_{1}\frac{x^{1}}{1!}\frac{y^{0}}{0!}\frac{z^{0}}{0!}\frac{h^{2}}{2!}\ \ = $$
$$\ \ \frac{x^{2}}{2!}z\ - \ xyh\ - \ 2x\frac{h^{2}}{2!}.$$
\end{footnotesize}
}
\end{exa}

We are ready to show the main result of this work, namely, an  explicit formula for the product on $ \ \widehat{U}_{h}(\mathfrak{sl}_2) \ $ in
the divided power basis. The formula itself may look unwieldy at first, but we show in Section \ref{bas} that it has a transparent combinatorial meaning. We denote it by  $\ \star \ $ the product on $ \ \widehat{U}_{h}(\mathfrak{sl}_2) \ $ in order to distinguish it
from the commutative product of formal power series.

\begin{thm}\label{mt}
{\em Let $\ f,g \ \in \ \widehat{U}_{h}(\mathfrak{sl}_2) \ $ be given by
$$\ f\ = \ \sum_{a,b,c,d  \in \mathbb{N}} f_{a,b,c,d}\frac{x^{a}y^{b}z^{c}h^{d}}{a!b!c!d!} \ \ \ \ \ \ \ \mbox{and} \ \ \ \ \ \ \ g\ = \
\sum_{k,l,m,n  \in \mathbb{N}} g_{k,l,m,n}\frac{x^{k}y^{l}z^{m}h^{n}}{k!l!m!n!}.\ $$
The product $\ \ f\star g \ \in \ \widehat{U}_{h}(\mathfrak{sl}_2)\ $  is given by
$$f\star g \ \ = \ \ \sum_{\alpha,\beta,\gamma,\rho \in \mathbb{N}} \ (f\star g)_{\alpha,\beta,\gamma,\rho} \
\frac{x^{\alpha} y^{\beta}z^{\gamma}h^{\rho}}{\alpha!\beta!\gamma!\rho!}$$
where  the coefficients $\ (f\star g)_{\alpha,\beta,\gamma,\rho} \ $  are defined using $\ 13 \ $ auxiliary variables $$\alpha_1, \ \alpha_2, \ \beta_1, \ \beta_2, \ \beta_3, \ \gamma_1, \ \gamma_2,\ \rho_1, \ \rho_2, \ \rho_3, \ \rho_4,\ \rho_5,\ \rho_6 \ \in \ \mathbb{N}^{13}\ \ \ \ \ \mbox{such that:} $$
 $$\alpha_1 \ + \ \alpha_2\ \ = \ \ \alpha, \ \ \ \ \ \ \beta_1 \ + \ \beta_2 \ + \ \beta_3 \ \ = \ \ \beta,
\ \ \ \ \ \ \gamma_1  \ + \ \gamma_2  \ \ = \ \ \gamma,$$
$$\rho_1 \ + \ \rho_2 \ + \ \rho_3 \ + \ \rho_4 \ + \ \rho_5 \ + \ \rho_6\ = \ \rho, \ \ \ \ \ \ \beta_{3}\ + \ \rho_{4}\ = \ \rho_{3}.$$

The constant $\ (f\star g)_{\alpha,\beta,\gamma,\rho}  \ $ is given by
$$\sum \ f_{\alpha_{1},\beta_{1}+\rho_{5},\gamma_{1}+\rho_{3},\rho_{1}} \ g_{\alpha_{2}+\rho_{3},\beta_{3}+\rho_{6},\gamma_{2},\rho_{2}} \ \ \binom{\alpha}{\alpha_{1},\alpha_{2}}\ \binom{\beta}{\beta_{1},\beta_{2},\beta_{3}} \ \binom{\gamma}{\gamma_{1},\gamma_{2}} $$
$$(-1)^{\rho_{3}}\ \binom{\rho}{\rho_{1},\rho_{2},\rho_{3},\rho_{4},\rho_{5},\rho_{6}} \ (2\alpha_{2})^{\rho_{5}} \ (2\gamma_{1})^{\rho_{6}}\ \beta_{2}!\ \rho_{4}!\ (\gamma_{1}+\alpha_{2})_{\beta_3 + \rho_{4}}^{\rho_{4}}\
, $$ where the broken line means multiplication.

}
\end{thm}

\begin{proof} The result follows after several applications of Proposition \ref{bd}. By definition the product $\ f\star g \ $ is equal to

$$ \sum_{a,b,c,d,k,l,m,n}^{} \;f_{a,b,c,d} \ g_{k,l,m,n}\left( \frac{x^{a}}{a!} \frac{y^{b}}{b!}\frac{z^{c}}{c!}\frac{h^{d}}{d!}\right)
 \star \left( \frac{x^{k}}{k!} \frac{y^{l}}{l!}\frac{z^{m}}{m!}\frac{h^{n}}{n!}\right).$$
Since $\ h \ $ commutes with the other variables, we put together the $\ h$-monomials and get
$$\sum_{}^{} \;\left[ f_{a,b,c,d}\ g_{k,l,m,n}\binom{d+n}{d,n} \right]  \frac{x^{a}}{a!} \frac{y^{b}}{b!} \left( \frac{z^{c}}{c!}  \frac{x^{k}}{k!} \right) \frac{y^{l}}{l!}\frac{z^{m}}{m!} \frac{h^{d+n}}{(d+n)!}.$$
Using Proposition \ref{bd} we order the selected $\ z \ $ and $\ x \ $ monomials obtaining

\

$$\sum_{}^{} \ \left[ f_{a,b,c,d}\ g_{k,l,m,n}\ \binom{d+n}{d,n}\ (-1)^{v}\ (v-w)!\ (v+w)_{w}\ (c+k-2v)_{v}^{w} \right]$$
$$ \frac{x^{a}}{a!} \frac{y^{b}}{b!} \left( \frac{x^{k-v}}{(k-v)!}  \frac{y^{v-w}}{(v-w)!}\frac{z^{c-v}}{(c-v)!}\frac{h^{v+w}}{(v+w)!} \right) \frac{y^{l}}{l!}\frac{z^{m}}{m!} \frac{h^{d+n}}{(d+n)!}$$

\

\noindent $\mbox{where} \ \ 0 \leq w \leq v \leq \mathrm{min}(c,k).\ \ $ We collet $\ h$-monomials together and get:
$$\sum_{}^{} \;\left[ f_{a,b,c,d}\ g_{k,l,m,n}\ \binom{d+n}{d,n}\ \binom{d+n+v+w}{d+n,\ v+w}\ (-1)^{v}\ (v-w)!\ (v+w)_{w}\ \big(c+k-2v \big)_{v}^{w} \right]$$
$$ \frac{x^{a}}{a!} \left( \frac{y^{b}}{b!}  \frac{x^{k-v}}{(k-v)!}\right)  \frac{y^{v-w}}{(v-w)!}\frac{z^{c-v}}{(c-v)!}\frac{y^{l}}{l!}\frac{z^{m}}{m!} \frac{h^{d+n+v+w}}{(d+n+v+w)!}.$$

\

\noindent Using Proposition \ref{bd} we write the selected $\ y \ $ and $ \ x \ $ monomials in normal order

$$\sum_{}^{} \;\left[ f_{a,b,c,d}\ g_{k,l,m,n}\ \binom{d+n+v+w}{d,n,v+w}\ (-1)^{v}\ (v-w)!\ (v+w)_{w} \ \big(c+k-2v\big)_{v}^{w} \ (2(k-v))^{i}\right]$$
$$ \frac{x^{a}}{a!} \left( \frac{x^{k-v}}{(k-v)!}\frac{y^{b-i}}{(b-i)!}  \frac{h^{i}}{i!}\right)  \frac{y^{v-w}}{(v-w)!}\frac{z^{c-v}}{(c-v)!}\frac{y^{l}}{l!}\frac{z^{m}}{m!} \frac{h^{d+n+v+w}}{(d+n+v+w)!}$$

\

\noindent $\mbox{with} \ \ 0 \leq i \leq b. \ $ Collecting $\ h$-monomials we get:

$$\sum_{}^{} \; f_{a,b,c,d} \ g_{k,l,m,n}\binom{d+n+v+w+i}{d,n,v+w,i}\binom{a+k-v}{a,k-v}\binom{b-i+v-w}{b-i,v-w}(-1)^{v}(v-w)!(v+w)_{w} $$
$$ \big(c+k-2v\big)_{v}^{w} (2(k-v))^{i} \frac{x^{a+k-v}}{(a+k-v)!} \frac{y^{b-i+v-w}}{(b-i+v-w)!}\left( \frac{z^{c-v}}{(c-v)!}\frac{y^{l}}{l!}\right) \frac{z^{m}}{m!} \frac{h^{d+n+v+w+i}}{(d+n+v+w+i)!}.$$

\

\

\noindent Using Proposition \ref{bd} to write the selected $\ z \ $ and $ \ y \ $ monomials in normal order, and collecting $\ h$-monomials we get:

$$\sum \ f_{a,b,c,d} \ g_{k,l,m,n}\binom{d+n+v+w+i+u}{d,n,v+w,i,u}\binom{c-v+m}{c-v,m}\binom{b-i+v-w+l-u}{b-i,v-w,l-u}$$

$$ (-1)^{v}\ \binom{a+k-v}{a,k-v}(v-w)!\ (v+w)_{w} \ \big(c+k-2v\big)_{v}^{w} \ (2(k-v))^{i}\ (2(c-v))^{u} $$

$$\frac{x^{a+k-v}}{(a+k-v)!} \frac{y^{b-i+v-w+l-u}}{(b-i+v-w+l-u)!} \frac{z^{c-v+m}}{(c-v+m)!}\frac{h^{d+n+v+w+i+u}}{(d+n+v+w+i+u)!}$$

\

\noindent $\mbox{with} \ \ 0 \leq u \leq l.$ $\ \ \mbox{Using} \ \ (v+w)_{w}\ = \ \binom{v+w}{w}w!, \ \ $ we obtain:

$$\sum \ f_{a,b,c,d} \ g_{k,l,m,n}\binom{d+n+v+w+i+u}{d,n,v,w,i,u}\binom{c-v+m}{c-v,m}\binom{b-i+v-w+l-u}{b-i,v-w,l-u}$$

$$ (-1)^{v}\ \binom{a+k-v}{a,k-v} (v-w)! \ w! \ \big(c+k-2v\big)_{v}^{w} \ (2(k-v))^{i} \ (2(c-v))^{u} $$

$$\frac{x^{a+k-v}}{(a+k-v)!} \frac{y^{b-i+v-w+l-u}}{(b-i+v-w+l-u)!} \frac{z^{c-v+m}}{(c-v+m)!}\frac{h^{d+n+v+w+i+u}}{(d+n+v+w+i+u)!}.$$

\

\noindent Finally performing the change of variables specified below we obtain that:

$$\sum_{}^{} \; f_{\alpha_{1},\beta_{1}+\rho_{5},\gamma_{1}+\rho_{3},\rho_{1}} \ g_{\alpha_{2}+\rho_{3},\beta_{3}+\rho_{6},\gamma_{2},\rho_{2}}\binom{\rho}{\rho_{1},\rho_{2},\rho_{3},\rho_{4},\rho_{5},
\rho_{6}}\binom{\gamma}{\gamma_{1},\gamma_{2}}\binom{\beta}{\beta_{1},\beta_{2},\beta_{3}}\binom{\alpha}{\alpha_{1},\alpha_{2}}$$
$$ \Big[ (-1)^{\rho_{3}}\ \beta_{2}! \ \rho_{4}! \ (\gamma_{1}+\alpha_{2})_{\rho_{3}}^{\rho_{4}} \ (2\alpha_{2})^{\rho_{5}}\ (2\gamma_{1})^{\rho_{6}}\Big]\ \frac{x^{\alpha}}{\alpha!} \frac{y^{\beta}}{\beta!} \frac{z^{\gamma}}{\gamma!}\frac{h^{\rho}}{\rho!},$$
where  $$\alpha_1=a,\ \ \ \alpha_2=k-v,\ \ \ \beta_1=b-i,\ \ \ \beta_2=l-u ,\ \ \ \beta_3=v-w, \ \ \ \gamma_1=c-v, \ \ \ \gamma_2=m,$$
$$\rho_1=d,\ \ \ \rho_2=n,\ \ \ \rho_3=v,\ \ \ \rho_4=w,\ \ \ \rho_5=i ,\ \ \ \rho_6=u, \ \ \ \beta_3 +\rho_4 = \rho_3.$$
\end{proof}

\section{Combinatorial approach towards $\ \widehat{U}_{h}(\mathfrak{sl}_2)$}\label{bas}

The formula for the product on $\ \widehat{U}_{h}(\mathfrak{sl}_2)\ $ given in Theorem \ref{mt} although explicit is quite hard to digest. The presence of $13$ auxiliary variables may suggest  that an intuitive understanding of this product is simply hopeless. In this section we argue on the contrary, and show that a sophisticated but easy to grasp combinatorial understanding of the product rule is indeed possible.\\

 We are going to phrase our results in the language of the theory of species introduced by Joyal \cite{j1, j2}. The reader may consult \cite{berg} for a comprehensive introduction to the subject, including a fairly exhaustive list of references.  Here we just briefly introduce the main ideas needed in order to make this work
reasonably self-contained, following the categorical approach developed in \cite{bd, ecd, re2}. \\

Let $\ \mathbb{B} \ $  be the category of finite sets and bijections, i.e. $\ \mathbb{B}\ $ is the underlying groupoid of the category $\ \mathrm{set} \ $ of finite sets and maps. For $\ d\geq 1,\ $ let $\ \mathbb{B}^d \ $ be the $\ d$-fold Cartesian product of $\ \mathbb{B} \ $ with itself.
Objects of $\ \mathbb{B}^d \ $ are $d$-tuples of finite sets, and may also be regarded as pairs $\ (a,f),\ $ where $\ a \ $ is a finite
set and $\ f:a \rightarrow [d]\ $ is a  map. A morphism from $\ (a, f) \ $ to $\ (b, g) \ $ is a
bijection $\ \alpha:a \rightarrow b \ $ such that $\ g\alpha =  f. \ $  We think of  $\ [d] \ $ as a set of $d$ colors, and $\ (a,f) \ $ as a colored set.\\

Let $\ C \ $ be a distributive category, meaning that $C$ comes with functors $\ \oplus \ $ and $ \ \otimes \ $ satisfying suitable axioms.
Coherence laws for such structures have been introduced by Laplaza \cite{l2, l1}. For the purposes of this work the reader may take  $\ C \ $ to be the category of finite sets and maps, or a category of finite dimensional vector spaces and linear transformations over a field. In the latter case $\ \oplus \ $ and $ \ \otimes \ $ are, respectively, direct sum and tensor product of vector spaces, whereas in the former case
$\ \oplus \ $ and $ \ \otimes \ $ are disjoint union and Cartesian product, and thus are denoted by $\ \sqcup \ $ and $\ \times \ $. \\

We also demand that a negative functor $\ -:C \rightarrow C \ $ be defined on $\ C,\ $ which is assumed  to come with natural isomorphisms:
$$-0 \ \simeq \ 0, \ \ \ \ \ \ \ -(a\oplus b)\ \simeq \ -a\oplus -b, \ \ \ \ $$
$$(-a)\otimes b \ \simeq \ - (a \otimes b) \ \simeq \ a\otimes (-b) \ \ \ \ \mbox{for} \ \ a,b \in C.$$

\

There is a simple mechanism, akin to Grothendick's construction of the group associated to an abelian monoid,  enhancing a distributive category into a distributive category with a negative functor.
Namely, given $\ C \ $ one considers the Cartesian product category $\ \mathbb{Z}_2\mbox{-}C  =  C\times C. \ $
The sum, negative,  and product functors on $\mathbb{Z}_2\mbox{-}C$
are given on objects $a_1, a_2, b_1, b_2\in C$ by
$$(a_1,a_2)\oplus (b_1,b_2)\ = \ (a_1\oplus b_1 \ , \ a_2\oplus
b_2),  \ \ \  -(a_1,a_2)\ = \ (a_2,a_1),$$ $$(a_1,a_2)\otimes
(b_1,b_2)\ = \ (a_1\otimes b_1 \ \oplus \ a_2\otimes b_2 \ \ , \ a_1\otimes
b_2\ \oplus \ a_2\otimes b_1),$$ and naturally extended
to morphisms. There is  an inclusion functor $\ \ i:C\ \rightarrow \ \mathbb{Z}_2\mbox{-}C \ \ $
given on objects and morphisms by $\ i(a) = (a,0) \ \ \mbox{and} \ \  i(f) =  (f,1_a).$\\

 Note that we have a functor $\ \mathrm{set} \rightarrow C \ $ which sends a finite set $x$ to the object $\ \bigoplus_{i \in x}k\ $ of $\ C. \ $ This functor allow us to transport combinatorial constructions  to $\ C. \ $ \\

Consider the category $\ [\mathbb{B}^d, C] \ $  of functors from $\ \mathbb{B}^d \ $ to  $\ C, \ $ and natural transformations as morphisms. There are several interesting structural functors on $\ [\mathbb{B}^d, C] \ $ some of which we proceed to briefly describe.
These structural functors on $C$ are inspired by the close relationship between
functors in $\ [\mathbb{B}^d, C] \ $ and formal power series in $d$-variables defined via generating series.\\

We assume that $\ C \ $ comes with a valuation map $\ | \ |: C \rightarrow R \ $ where $\ R\ $ is a ring with $\ R \supseteq \mathbb{Q}. \ $  For example, for the category of finite sets we let  $|x|$ be the cardinality of $x$, and for the category of finite dimensional vector spaces we let $|V|$ be the dimension of $V$. The valuation map should satisfy the following identities for $\ a,b \in C:$
$$|a|=|b| \ \ \ \mbox{if} \ \ \ a \simeq b, \ \ \ \ \  |a\oplus b|=|a|+|b|, \ \ $$
$$|a \otimes b| = |a||b|, \ \ \ \ |1|=1, \ \ \ \ |0|=0, \ \ \ \  \mbox{and}  \ \ \ \
|-a|=-|a|.$$ The valuation map on $\ C \ $ can be extended to a map   $$| \ |: [\mathbb{B}^d, C] \ \rightarrow \ R[[x_1,...,x_d]] \ \ \ \ \ \mbox{which sends a functor to its generating series}$$
$$|F| \ = \ \sum_{(n_1,...,n_d) \in \mathbb{N}^d} |F([n_1],...,[n_d])|\frac{x_1^{n_1}...x_d^{n_d}}{n_1!...n_d!}.$$
Next we introduce  further structures present on the category $\ [\mathbb{B}^d, C].$

\begin{itemize}

\item The {\bf sum functor} $\ F + G  \in  [\mathbb{B}^d, C] \ $ is given on $\ (x,f) \in \mathbb{B}^d\ $ by
$$(F+G)(x,f) \ \ = \ \ F(x,f) \ + \ G(x,f),$$
and is such that $\ |F+G|  =  |F| + |G|.$

\item The {\bf negative functor} $\ -F \in  [\mathbb{B}^d, C] \ $ is given on $\ (x,f) \in \mathbb{B}^d\ $ by
$$(-F)(x,f) \ \ = \ \ -\big( F(x,f)\big),$$
and is such that $\ |-F|  =  -|F|.$

\item The {\bf Hadamard product} functor  $\ F\times G  \in  [\mathbb{B}^d, C] \ $ is given on $\ (x,f) \in \mathbb{B}^d\ $ by
$$F\times G(x,f) \ \ = \ \ F(x,f)\otimes G(x,f),$$
and is such that $\ |F\times G|  =  |F|\times |G|\ $ where the Hadamard product on series is given by coefficient-wise multiplication.

\item The {\bf product} functor $\ FG  \in  [\mathbb{B}^d, C] \ $ is given on $\ (x,f) \in \mathbb{B}^d\ $ by
$$FG(x,f) \ \ = \ \ \underset{(x_1,f_1) \sqcup (x_2,f_2)=(x,f)}{\bigoplus} \ F(x_1,f_1)\otimes G(x_2,f_2),$$
and is such that $\ |FG|  =  |F||G|. \ $

\item The {\bf composition} or {\bf substitution} functor $\ F(G_1,...,G_d)  \in  [\mathbb{B}^d, C] \ $ is  given on $\ (x,f) \in \mathbb{B}^d\ $ by
$$F(G_1,...,G_d)(x,f) \ \ = \ \ \underset{\underset{c: \pi \rightarrow [d]}{\pi \in \mathrm{Par}(x)}}{\bigoplus}
\ F(\pi,c)  \otimes  \bigotimes_{a \in \pi}G_{c(a)}(a,f|_a),$$ where we assume that $\ G_i(\emptyset) = 0. \ $ We have that
$\ |F(G_1,...,G_d)|  =  |F|(|G_1|,...,|G_d|) .$

\item A {\bf  quantum $\star$-product} on a suitable category of functors was introduced in \cite{re2} in order to produce a categorication of the formal homogeneous Weyl algebras. Explicitly, for $\ d \geq 1,\ $ consider the category of functors $$ [\mathbb{B}^{2d+1}, C] \ = \ [\mathbb{B}^{d}\times \mathbb{B}^{d}\times \mathbb{B}, C]. $$
We regard objects of $\ \mathbb{B}^{d}\times \mathbb{B}^{d}\times \mathbb{B} \ $ as triples $\ (x,f,h)\ $ where $x$ and $h$ are finite sets,
and $\ f: x \rightarrow [d]\sqcup [\widetilde{d}]\ $ is a map. The $\ \star$-product $\ F\star G \in [\mathbb{B}^{2d+1}, C] \ $ of
functors $\ F,G \in [\mathbb{B}^{2d+1}, C] \  $ is given by
$$F\star G(x,f,h)=\bigoplus F
(x_1 \sqcup h_3,f|_{x_1}\sqcup \widetilde{g},h_1)\otimes G(x_2
\sqcup h_3,f|_{x_2}\sqcup g,h_2)$$ where the sum runs
over all pairs $\ x_1, \  x_2 \ $ and all triples $\ h_1,\ h_2, \ h_3\ $ such
that $$x_1\ \sqcup  \ x_2\ = \ x,\ \ \ \ \  h_1 \ \sqcup \ h_2 \ \sqcup \ h_3\ = \  h \ \ \
\mbox{ \ \ and \ \ } \ \ \ g:h_3 \longrightarrow [d].$$
Where $\ f|_{x_1} \ $ and $\ f|_{x_2} \ $ are the restriction maps, and $\ \widetilde{g} \ $ has the same domain as $g$, and assumes the corresponding values in $\ [\widetilde{d}] = \{\widetilde{1},..,\widetilde{d} \}.\ $
This product is a categorification of the formal homogeneous Weyl algebra in the sense that we have a map
$$| \ |: [\mathbb{B}^{2d+1}, C]\ \longrightarrow \ \mathbb{W}_n(R)$$ sending a functor to its generating function $\ |F| \ $ given by
$$\sum_{(n_1,...,n_d,m_1,...,m_d,k) \in \mathbb{N}^{2d+1}} \big|F([n_1],...,[n_d],[m_1],...,[m_d],[k])\big|\frac{x_1^{n_1}...x_d^{n_d}}{n_1!...n_d!}
\frac{y_1^{m_1}...y_d^{m_d}}{m_1!...m_d!}\frac{h^{k}}{k!},$$
which satisfies that $$|F\star G| \ = \ |F|\star |G| .$$ On the right hand side of this equation the $\star$-product of formal power series is the
 product of formal power series in the formal homogeneous Weyl algebra:

$$\mathbb{W}_n(R)\ = \ R<<x_1,....,x_d,y_1,....,y_d, h>>/I_d, \ \  \mbox{where:} $$

\begin{itemize}
\item $R<<x_1,....,x_d,y_1,....,y_d,h>>\ $ is the ring of formal power series with coefficients in $\ R\ $ in the non-commutative
 variables $\ x_1,....,x_d,y_1,....,y_d,h$.

\item $I_d\ $ is the ideal generated by the  relations:
\begin{itemize}
\item $h\ $ commutes with all other variables.
\item  $x_jx_i=x_ix_j \ \ $ and  $\ \ y_jy_i=y_iy_j, \ \ $ for $\ i,j \in [d].$
\item $y_jx_i=x_iy_j, \ \ $ for $\ i\neq j \in [d]$.
\item $y_ix_i=x_iy_i + h, \ \ $ for $\ i \in [d].$
\end{itemize}
\end{itemize}

\end{itemize}

Next we introduce a few examples of functors that will be needed in our construction of the combinatorial counterpart of the product on $\ \widehat{U}_{h}(\mathfrak{sl}_2).\ $

\begin{itemize}

  \item  The functor of linear orderings $\ \mathbb{L} : \mathbb{B} \rightarrow \mathrm{set}\ $ is given on objects by
  $$\mathbb{L}(x) \  = \  \big\{\alpha:[x] \rightarrow x \ \big| \ \alpha\ \  \mbox{bijective} \big\}.$$
  The generating series of $\ \mathbb{L} \ $ is   $\ \ \displaystyle \frac{1}{1-x} .$

  \item The functor of maps $\ \big[\ , \ \big]:\mathbb{B}^2 \rightarrow \mathrm{set} \ $ is given on objects by
  $$\big[x,y\big] \ \ = \ \ \{f\ | \ f: x \rightarrow y \ \ \mbox{is a map}\}.$$ The generating series of $\ \big[\ , \ \big] \ $
  is given by $\ \displaystyle \big|[ \ , \ ] \big| \ = \ \displaystyle e^{ye^x}. $

  \item The functor  $\ \left(\bullet \right)_{\bullet}^{\bullet} : \mathbb{B}^3 \rightarrow \mathrm{set}\ $ is such that
  $\ (x)_y^z \ $ is the subset of $$ \big[\ [z]\ , \bigsqcup_{i=0}^{|y|-1} x \sqcup [i]\ \big] $$
   consisting of those maps such that:
   \begin{itemize}
     \item For $\ s< t\ $ in  $\ [z], \ $  if $\ s \ $ is mapped into the component   $\ x \sqcup [i], \ $ then $\ t \ $ is mapped into a component
     $\ x \sqcup [j] \ $ with $\ i < j.$
     \item By convention $\ (x)_{\emptyset}^{\emptyset} \ $ is a one element set, and  $\ (x)_{\emptyset}^{z}=\emptyset \ $
   for $\ z \neq \emptyset. $
   \end{itemize}

   The generating function of $\ \left(\bullet \right)_{\bullet}^{\bullet} \ $ is given by
   $$\big|\left(\bullet \right)_{\bullet}^{\bullet} \big| \ = \ \sum_{n, m,  k \in \mathbb{N}}(n)_m^k\frac{x^ny^mz^k}{n!m!k!}. $$

   \item Below we  use the functor $\ \mathbb{B}^4 \rightarrow C \ $  sending a tuple of finite sets $\ (a,b,c,d) \ $ to
     $\ \left(a \sqcup b \right)_{c \sqcup d}^{d}. $

   \item The $x$-colored singleton functor $ X: \mathbb{B}^4 \rightarrow C $ evaluated on  $\ (x,y,z,h) \ $ is $1$ if $\ (|x|,|y|,|z|,|h|) = (1,0,0,0), \ $ and $0$ otherwise. Singleton functors $\ Y, \ Z, \ $ and  $\ H\ $ are defined analogously. The corresponding generating series are given by
       $\ |X|=x, \ |Y|=y, \ |Z|=z, \ |H|=h.$

   \item  The divided power functor $\ \frac{X^{a}}{a!}: \mathbb{B}^4 \rightarrow \mathbb{C}, \ $ for $\ a \geq 0, \ $ is given by
$$\frac{X^{a}}{a!}(x,y,z,h) \ = \  \left\{
\begin{array}{cc}
1 & \mathrm{if}\ \ \ |x|=a, \ y=z=h=\emptyset,\\
0 & \ \mathrm{otherwise},
\end{array}\right. $$
with analogue definitions for the divided powers functors   $\ \frac{Y^{a}}{a!}, \ \frac{Z^{a}}{a!}, \ \frac{H^{a}}{a!}$. The corresponding generating series are given by  $\ |\frac{X^{a}}{a!}|=\frac{x^{a}}{a!}, \ |\frac{Y^{a}}{a!}|=\frac{y^{a}}{a!}, \
|\frac{Z^{a}}{a!}|=\frac{z^{a}}{a!}, \  |\frac{H^{a}}{a!}|=\frac{h^{a}}{a!}$.

\end{itemize}

Recall that we are fixing a distributive category with negative objects $C$ which comes with a valuation taking values in a ring $\ R \supseteq \mathbb{Q}.\ $
Our next goal is to introduce a $\ \star$-product on the category $\ [ \mathbb{B}^4, C] \ $ which encodes the combinatorial properties of the product
on $\ \widehat{U}_{h}(\mathfrak{sl}_2). \ $ We think of an object in $\ \mathbb{B}^4 \ $ either as a $4$-tuple of finite sets, or as a $4$-colored finite
set with the following conventions: red represents the variable $\ x, \ $ purple represents the variable $\ y, \ $ green represents the variable $\ z, \ $ and blue represents the variable $\ h. $

\begin{defi}\label{fp}
{\em Let $\ F,G \ $  be functors $\ \mathbb{B}^4 \rightarrow C, \ $ the $\star$-product functor $\ F \star G :\mathbb{B}^4 \rightarrow C\ $  sends a
tuple $\ (x,y,z,h) \in \mathbb{B}^4 \ $  to the object  $\ F \star G (x,y,z,h)  \in C \ $ given by
$$ \displaystyle \bigoplus\left( -1 \right)^{\mid h_3 \mid} F\left( x_1, y_1 \sqcup h_5, z_1 \sqcup h_3, h_1 \right)\otimes G\left( x_2 \sqcup h_3, y_2 \sqcup h_6, z_2, h_2 \right) \otimes \mathbb{M}(x_2,z_1,y_3,h_4,h_5,h_6)$$ where the sum runs over the tuples
$$\ x_1,\  x_2, \ y_1, \ y_2, \ y_3,\  z_1, \ z_2, \  h_1, \ h_2, \ h_3,\ h_4,\ h_5,\  h_6 \  \ \in  \ \mathbb{B}^{13} \ $$ such that
$$x_1 \sqcup x_2\ = \ x, \ \ \ \ \ \  \ \ \ y_1 \sqcup y_2 \sqcup y_3\ = \ y, \ \ \ \ \ \ \ \ \ z_1 \sqcup z_2\ = \ z,$$
$$h_1 \sqcup h_2 \sqcup h_3\sqcup h_4 \sqcup h_5 \sqcup h_6\ = \ h, \ \ \ \ \ \ \ \ \ |y_3|  \ + \  |h_4| \ =  \ h_3,$$
and the functor $\ \mathbb{M}:\mathbb{B}^6  \rightarrow  C\ $ is given on objects by
$$\mathbb{M}(x_2,z_1,y_3,h_4,h_5,h_6) \  =  \ [h_5, x_2 \sqcup x_2] \otimes  [h_2,z_1 \sqcup z_1] \otimes \mathbb{L}(y_3) \otimes \mathbb{L}(h_4) \otimes
(z_1 \sqcup x_2)_{y_3 \sqcup h_4}^{y_3}.$$
}
\end{defi}

\begin{prop}{\em The product $\ F \star G \ $ is indeed a functor $\ \mathbb{B}^4 \rightarrow C.$
}
\end{prop}

\begin{proof}We must show that a $4$-tuple of bijections from $\ (x, y, z,h) \ $ to $\ (a,b,c,d) \ $ induces a map
$ \ F \star G(x, y, z,h) \rightarrow  F \star G(a,b,c,d), \ $  and that this correspondence respects composition.
The result follows since the structures involved in the construction of $\ F\star G \ $(partitions of sets, maps between sets, linear orderings, $F$-structures, $G$-structures, and $\ \left(\bullet \right)_{\bullet}^{\bullet}$-structures) are transportable under bijections.
\end{proof}

The category $\ [ \mathbb{B}^4, C]  \ $ comes equipped with a natural map
$$| \ |: [ \mathbb{B}^4, C] \ \longrightarrow \ \widehat{U}_{h}(\mathfrak{sl}_2) \ \ \ \ \ \ \mbox{given by}$$
$$|F| \ = \  \sum_{a,b,c,d \in \mathbb{N}}|F([a],[b],[c],[d])|\frac{x^ay^bz^ch^d}{a!b!c!d!}.$$

\begin{figure}[t]
 \begin{tabular}{|p{15cm}|p{15cm}|p{15cm}|}\hline
\begin{center}
\includegraphics[width=15cm]{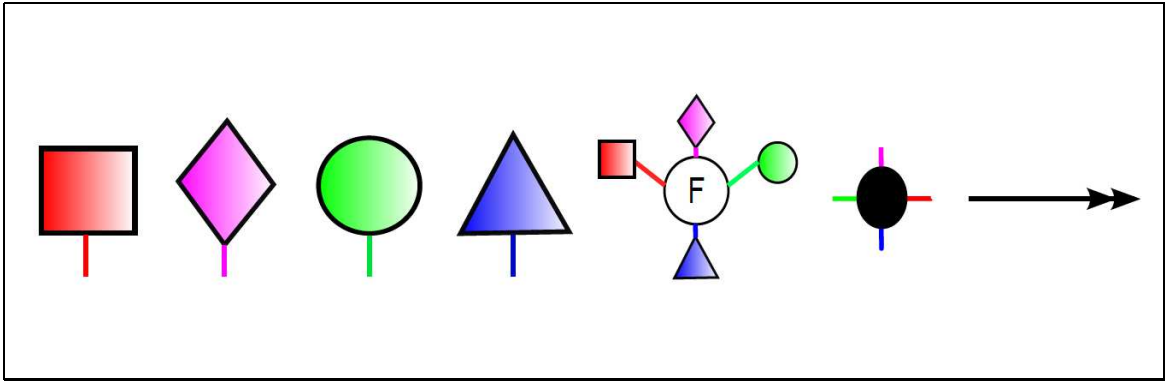}
\end{center}
\\\hline
\end{tabular}
\caption{Basic ingredients for the graphical interpretation.}\label{f2}
\end{figure}

\begin{thm}\label{im}
{\em For $\ F,G \in [ \mathbb{B}^4, C] \ $ we have that $\ \ \big| F \star G\big| \ = \ |F|\star |G|.$
}
\end{thm}

\begin{proof}The result follows by correlating the various ingredients taking part in Definition \ref{fp} and Theorem  \ref{mt}:
\begin{itemize}
  \item The partitions in Definition \ref{fp} give rise to the sum, the binomial coefficients, and the multinomial coefficients in Theorem \ref{mt}. The same sign $\ (-)^{|h_3|} = (-1)^{\rho_3}\ $ is applied in both cases.
  \item The $\ \otimes$-factors $\  F\left( x_1, y_1 \sqcup h_5, z_1 \sqcup h_3, h_1 \right)\otimes G\left( x_2 \sqcup h_3, y_2 \sqcup h_6, z_2, h_2 \right)\ $ in  Definition \ref{fp} give rise to the factors $\ f_{\alpha_{1},\beta_{1}+\rho_{5},\gamma_{1}+\rho_{3},\rho_{1}} \ g_{\alpha_{2}+\rho_{3},\beta_{3}+\rho_{6},\gamma_{2},\rho_{2}} \ $ in  Theorem \ref{mt}.
  \item The $\ \otimes$-factors $\ [h_5, x_2 \sqcup x_2] \otimes  [h_6,z_1 \sqcup z_1] \ $ in  Definition \ref{fp} give rise to the factors $\ (2\alpha_{2})^{\rho_{5}} \ (2\gamma_{1})^{\rho_{6}} \ $ in Theorem \ref{mt}.
  \item The $\ \otimes$-factors  $\ \mathbb{L}(y_3) \otimes \mathbb{L}(h_4)\ $ in  Definition \ref{fp} give rise to the factorial factors
  $\ \beta_{2}!\ \rho_{4}!\ $  in Theorem \ref{mt}.
  \item  The $\ \otimes$-factor $ \ (z_1 \sqcup x_2)_{y_3 \sqcup h_4}^{y_3} \ $ in  Definition \ref{fp} gives rise to factor
   $\ (\gamma_{1}+\alpha_{2})_{\beta_{3}+\rho_{4}}^{\rho_{4}}\ $   in Theorem \ref{mt}.
\end{itemize}
\end{proof}

\section{Graphs and the $ \star $-product on $\ [ \mathbb{B}^4, C]$}

In this final  section  we introduce a graphical interpretation for the $\star$-product on $\ [ \mathbb{B}^4, C], \ $ and thus we obtain, via Theorem \ref{im}, a combinatorial interpretation for the product on $\ \widehat{U}_{h}(\mathfrak{sl}_2). \ $ Let us introduce the basic
ingredients,  shown in Figure \ref{f2}, from which we construct the kind of graphs that we are going to need.

\begin{itemize}
\item The red square with a line attached to it represents a $x$-colored set. If one wishes to be more specific we draw as many red lines as elements are in the set.  Generically we draw only one line which stands for a multiplicity of lines. The same remark applies to the other basic components of our graphical constructions.

\item The purple diamond with a line attached to it represents a $y$-colored set.

\item The green disk with a line attached to it represents a $z$-colored set.

\item The blue triangle with a line attached to it represents a $h$-colored set.

\item The blob marked with the functor $F:\mathbb{B}^4  \rightarrow C$  represents the application of  $F$ to the disjoint union of the colored sets attached to it. Note that the color, the kind of figures, and even the position of attachment indicates the kind  of variable represented by the various elements attached to the blob $\ F. \ $ Figure \ref{f3} shows the diagram representing an application of the functor $\ F\star G$.

\begin{figure}[t]
\centering
\begin{tabular}{|p{5cm}|p{5cm}|p{5cm}|}\hline
\begin{center}
\includegraphics[width=5cm]{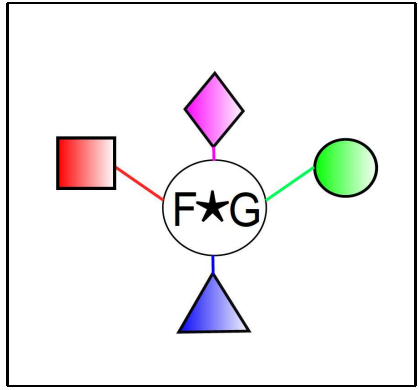}
\end{center}
\\\hline
\end{tabular}
\caption{Diagrammatic representation of an application of the $\star$-product $F\star G$.}\label{f3}
\end{figure}

\item The black disk represents the application of the functor $\ \mathbb{L}(y_3) \otimes \mathbb{L}(h_4) \otimes
(z_1 \sqcup x_2)_{y_3 \sqcup h_4}^{y_3}\ $ to the sets attached to it, considered as an ordered tuple of sets using the counter clockwise cyclic order and starting from the set attached at the West position.
\item A double arrow line represents the applications of the functor $\ \mathbb{B}^2  \rightarrow C\ $ sending $\ (a,b) \ $ to  $\ [a,b \sqcup b], \ $ where $\ a \ $ and $\ b\ $ are the incoming and outgoing sets linked by the arrow.
\end{itemize}

\begin{prop}{\em The product $\ F\star G\ $ is  represented, explicitly, by the graph in Figure \ref{f4}.
}
\end{prop}

\begin{figure}[t]
\centering
\begin{tabular}{|p{15cm}|p{1cm}|p{1cm}|}\hline
\begin{center}
\includegraphics[width=15cm]{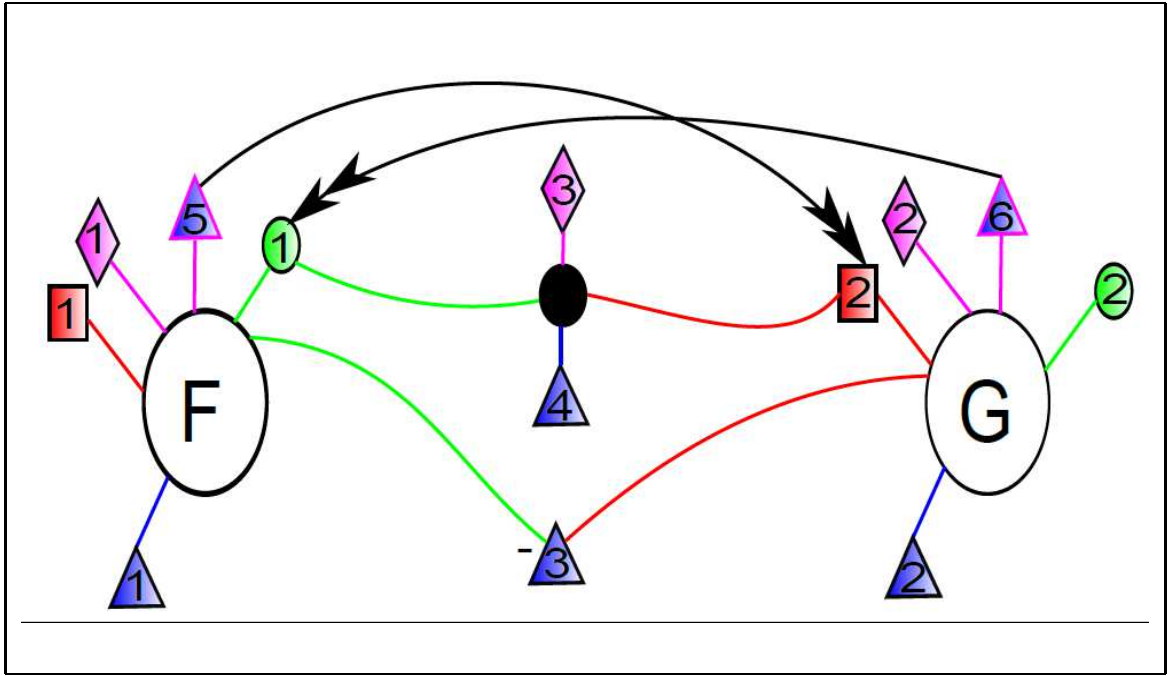}
\end{center}
\\\hline
\end{tabular}
\caption{Graphical representation of the $\star$-product functor $F\star G$.}\label{f4}
\end{figure}

\begin{proof} We have to check that the various components of the graph from Figure  \ref{f4} are in correspondence with
the $\ \otimes$-terms of the product $\ F\star G\ $ as given in Definition  \ref{fp}. We proceed to analyze the various components
of our graph.
\begin{itemize}
\item The red squares numbered $\ 1 \ $ and $\ 2 \ $ represent the partition of the $x$-colored set in two blocks.
\item The purple diamonds numbered $\ 1, \ 2 \ $ and $\ 3 \ $ represent the partition of the $y$-colored set in three  blocks.
\item The green disks numbered $\ 1 \ $ and $\ 2 \ $ represent the partition of the $z$-colored set in two blocks.
\item The blue triangles numbered $\ 1 \ $ trough $\ 5 \ $ represent the partition of the $h$-colored set in five blocks.
\item The blob marked by $\ F \ $ represents the application of the functor $\ F \ $ to the sets attached to it, which are $\ x_1, \ y_1 \sqcup h_5 \ $
(where the block $\ h_5 \ $ changes from a blue $z$-color to a purple $y$-color), $\ z_1 \sqcup h_3, \ $ and $\ h_1$.
Note that the block $\ h_3\ $ becomes a set of  bi-colored edges starting as green $y$-edges and ending up as a red $x$-edges.
\item The blob marked $\ G \ $ represents the application of the functor $\ G \ $ to the sets attached to it, namely $\ h_3 \sqcup x_2, \ y_2 \sqcup h_2  \ $
(again the block $\ h_2\ $ changes from a blue $z$-color to a purple $y$-color), $\ z_2, \ $ and $\ h_2$.
\item The black disk with the various edges attached to it represents the application of the functor $\ \mathbb{L}(y_3) \otimes \mathbb{L}(h_4)\otimes (z_1 \sqcup x_2)_{y_3 \sqcup h_4}^{y_3}.$
\item The double pointed arrows represent $\ [h_5,x_2 \sqcup x_2] \ $ and $\ [h_2, z_1 \sqcup z_1], \ $ respectively.
\item The negative sign comes from the cardinality of the block $\ h_3.$
\item The condition $\ |y_3|  \ + \  |h_4| \ =  \ |h_3|, \ $ implies that if $\ h_3 \ $ is empty, then  $\ y_3 \ $ and $\ h_4\ $ are also empty;
and that if the block $\ h_3 \ $ is not empty, then $\ y_3 \ $ and $\ h_4 \ $ can not be both empty.
\end{itemize}

\end{proof}

Next we put the graphical notation in action, thereby showing that it is an effective computational tool.
Recall that the colored singular functors $\ X,Y,Z, H\ $ output $\ 0 \ $ unless applied to a set of cardinality $\ 1 \ $ of the respective color where
it outputs $\ 1.$

\subsection{Graphs and the defining identities of $\ \widehat{U}_{h}(\mathfrak{sl}_2) \ $}

Let us study  the graphical interpretation of the defining identities for $\ \widehat{U}_{h}(\mathfrak{sl}_2). \ $

\begin{prop}{\em Consider the singular functors $\ X, Y, Z, H \in [ \mathbb{B}^4, C] .\ $ We have that
\begin{enumerate}
  \item  The functor $\ Y\star X \ $ is given by  $$Y\star X(x,y,z,h)  \ = \  \left\{
\begin{array}{cc}
1 & \mathrm{if} \ \ \ |x|=|y|=1, \ z=h=\emptyset,\\
1 & \mathrm{if} \ \ \ |x|=|h|=1, \ y=z=\emptyset,\\
0 &  \mbox{otherwise}.
\end{array}\right. $$ Therefore we have that $\ y\star x \ = \ |Y\star X|\ = \ xy + 2xh.$

  \item The functor $\ Z\star X \ $ is given by  $$Z\star X(x,y,z,h)  \ = \   \left\{
\begin{array}{cc}
1 & \mathrm{if} \ \ \ |x|=|z|=1, \ y=h=\emptyset,\\
-1 & \mathrm{if} \ \ \ |y|=|h|=1, \ x=z=\emptyset,\\
0 &  \mbox{otherwise}.
\end{array}\right. $$ Therefore we have that $\ z\star x \ = \ |Z\star X| \ = \ xz - yh.$

\item The functor $\ Z\star Y \ $ is given by  $$Z\star Y(x,y,z,h) \ = \ \left\{
\begin{array}{cc}
1 & \mathrm{if} \ \ \ |b|=|c|=1, \ a=d=\emptyset,\\
2 & \mathrm{if} \ \ \ |b|=|d|=1, \ a=c=\emptyset,\\
0 &  \mbox{otherwise}.
\end{array}\right. $$ Therefore we have that $\ z\star y \ =  \ |Z\star Y| \ = \ yz + 2yh.$

\end{enumerate}
}
\end{prop}

\begin{proof} The reader should have Figure \ref{f5} in mind as we develop our arguments.

\begin{figure}[t]
\centering
\begin{tabular}{|p{4.5cm}|p{4.5cm}|p{4.5cm}|}\hline
\begin{center}
\includegraphics[width=4.5cm]{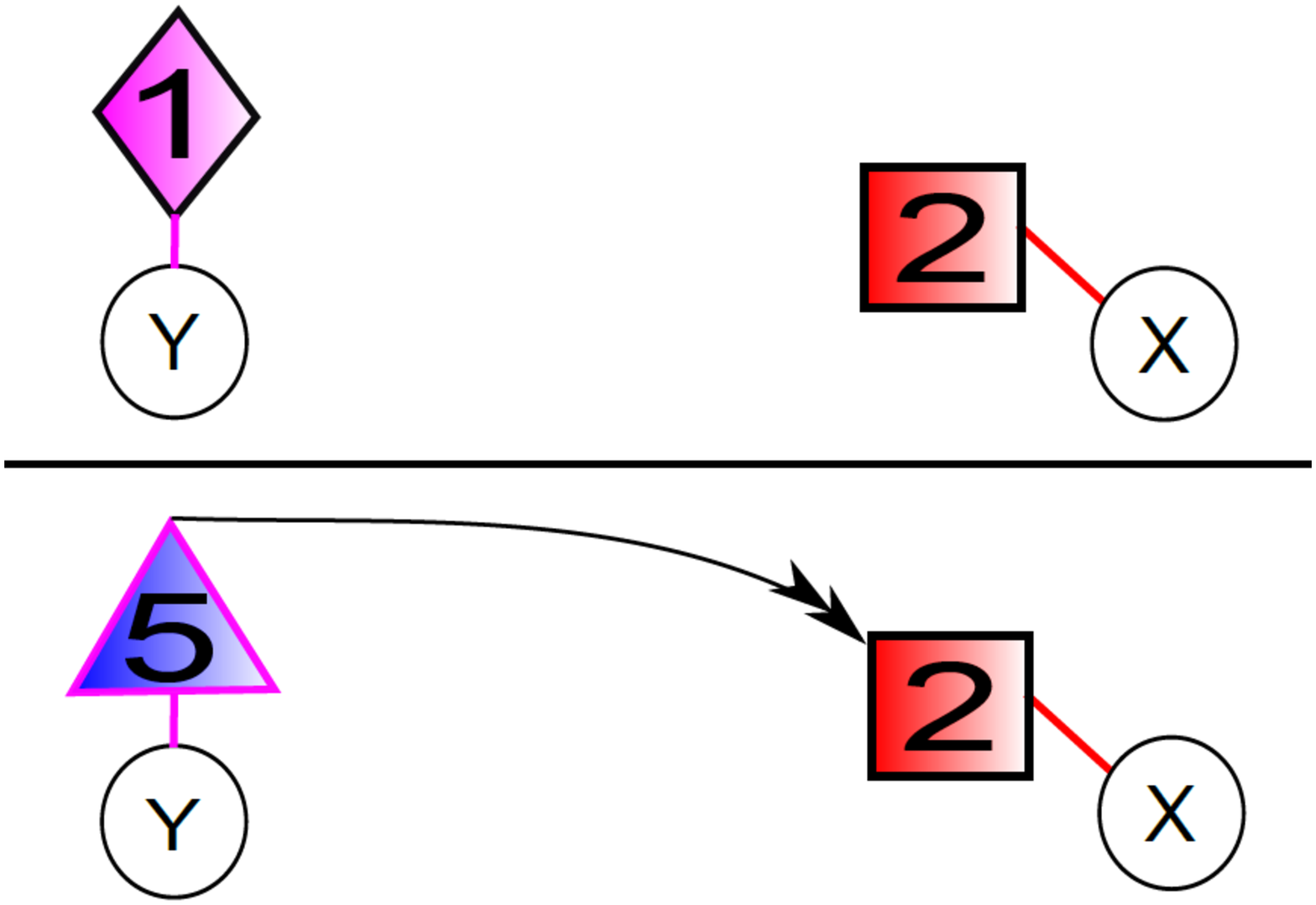}
\par\textbf{$|Y\star X|= xy  + 2xh.$}
\end{center}
&
\begin{center}
\includegraphics[width=4.5cm]{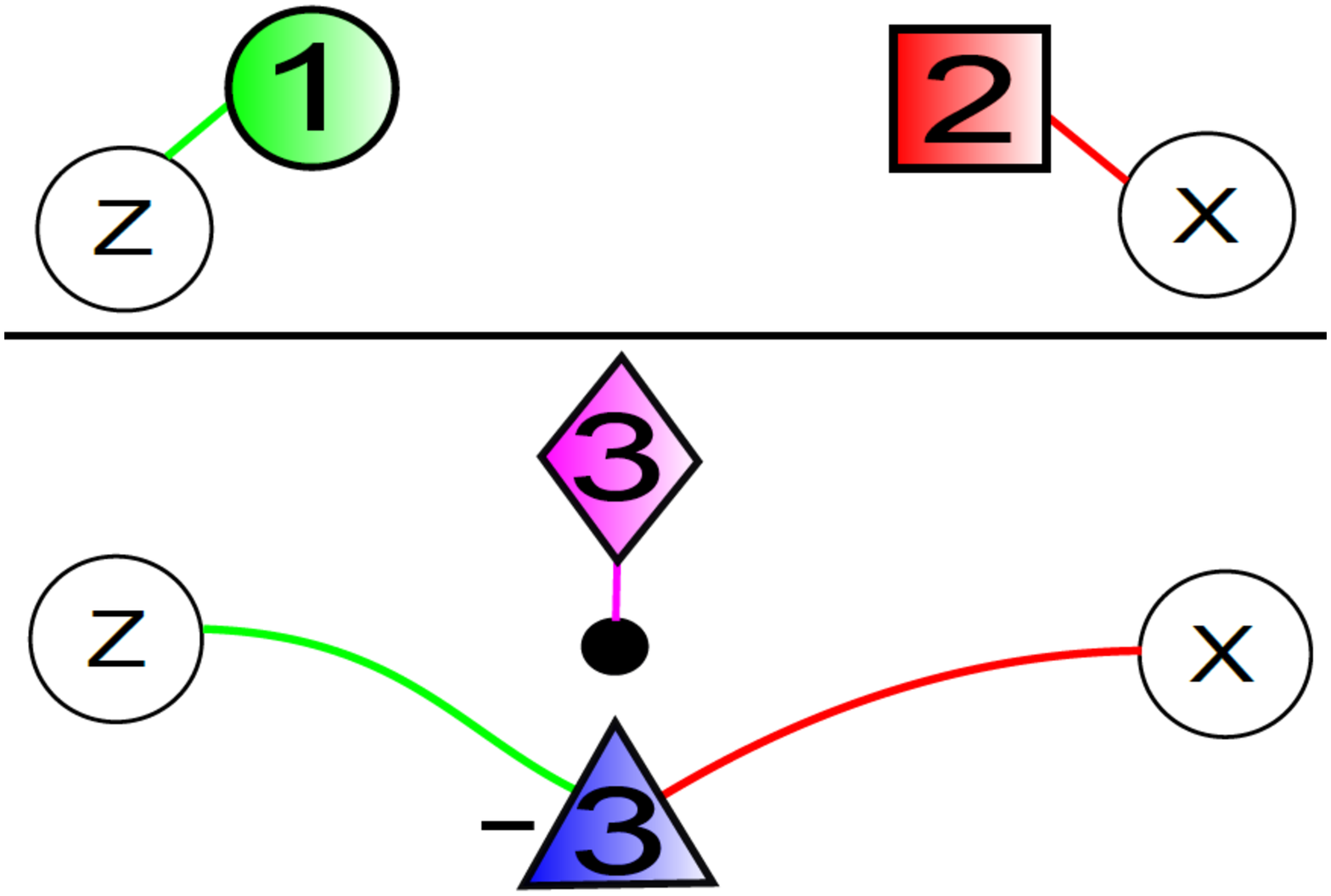}
\par\textbf{$|Z\star X|= xz - yh.$}

\end{center}
&
\begin{center}
\includegraphics[width=4.5cm]{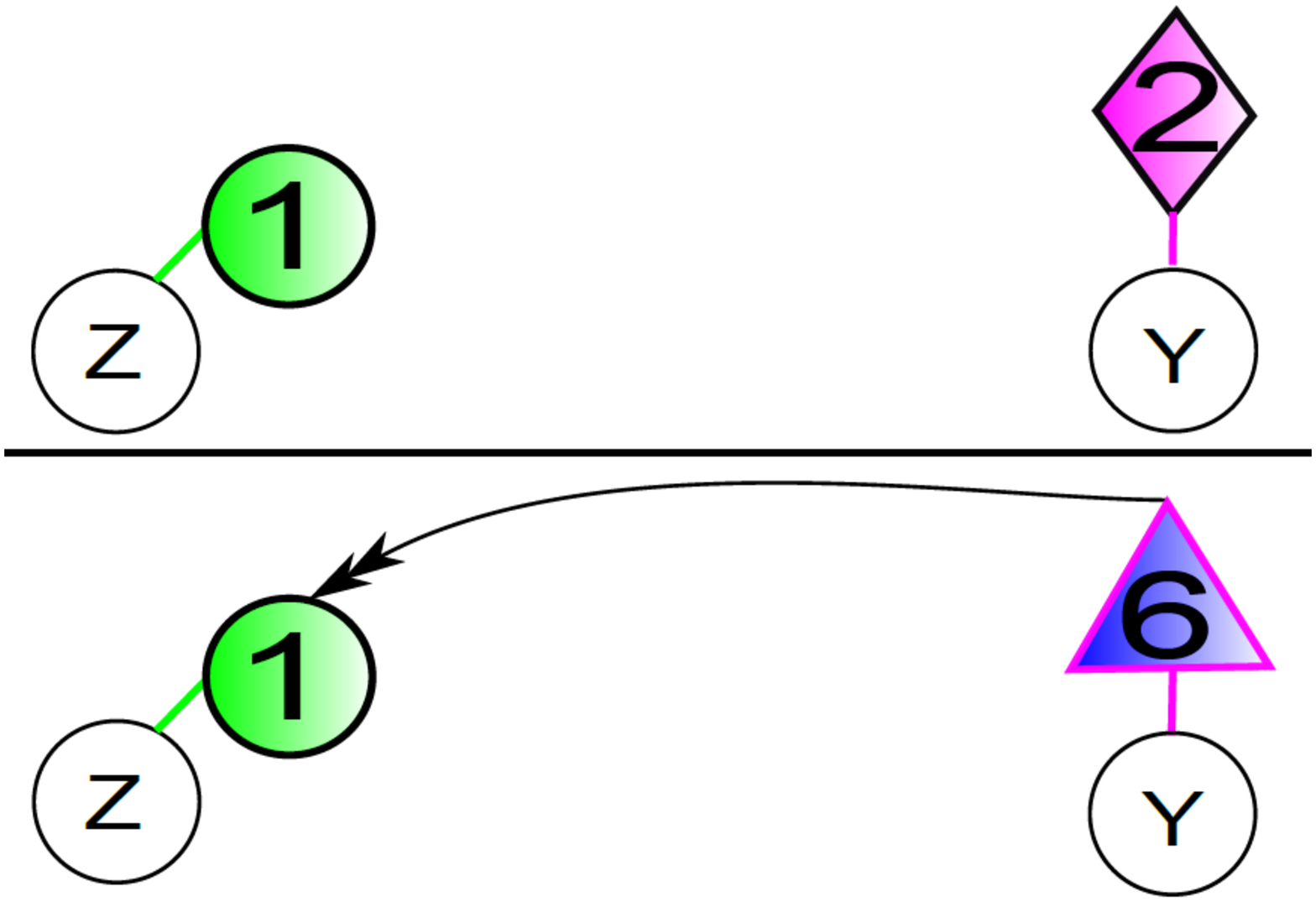}
\par\textbf{$|Z\star Y|= yz + 2zh.$	}

\end{center}
\\\hline
\end{tabular}
\caption{Graphical representation for the defining identities of $\ \widehat{U}_{h}(\mathfrak{sl}_2). \ $ }\label{f5}
\end{figure}

\begin{enumerate}

\item  By the definition of the singular species there must be exactly one $y$-element attached to the $Y$-blob, and exactly one
$x$-element attached to the $X$-blob.  This can happen in two ways: either we originally have the required elements, or we had an $h$-element and
a $x$-element. The $h$-element falls into the third block of the decomposition of $\ h \ $ and thus becomes the needed $y$-element. In this case we must also consider the maps from $\ h_5 \ $  to the disjoint union of two copies of the  $x$-set, yielding the required factor of $\ 2$.\\

\item Again we have two cases: either we have an $x$-element and a $z$-element, or we have an $h$-element and a $y$-element.
The $h$-element falls in block number $\ 3$ \  and fills the place of the $\ z \ $ and $\ x\ $ elements needed. The $y$-element goes to the third block, its attached to the black disk yielding a factor of $\ 1.$ The contribution of a graph with $\ |h|\geq 2 \ $ is equal to $\ 0.\ $ Indeed, the only active blocks for the partition of $\ h \ $ are $\ h_3\ $ and $\ h_4,\ $ and we know that $\ |h_3|\leq 1\ $ (otherwise the applications of the functor $\ Z\ $ yields a $\ 0 \ $ factor.)
Now if $\ |h|\geq 3, \ $ then $\ |h_4| = |h| -|h_3| \geq 2 > |h_3|, \ $ a contradiction since we know that $\ |h_4|\leq |h_3|. \ $
If $\ |h|=2\ $ and $\ |h_3|=|h_4|=1,\ $ then $\ |y_3|=0 \ $ and $\  (z_1+x_2)_{y_3+h_4}^{h_4}   \ $ gives rise to a factor of  $\ (0+0)_{0+1}^{1} =0. $ \\

\item   There must be exactly one $z$-element attached to the $Z$-blob, and exactly one
$x$-element attached to the $X$-blob.  So either we are given the required elements, or we had an $h$-element and
a $x$-element. The $h$-element necessarily falls into the block $\ h_6 \ $ and thus becomes the required $y$-element. We must also consider the maps from $\ h_6 \ $  to the disjoint union of two copies of the  $z$-set, yielding the required factor of $\ 2$.

\end{enumerate}

\end{proof}

\subsection{Graphical representation of the identities from Lemma \ref{bd}}

In this subsection  we study the graphical representation of the identities in Lemma \ref{bd}. We begin with an example, namely, we consider
the $\star$-product
$$\frac{Z^{2}}{2!}\star \frac{X^{2}}{2!}$$
for which we adopt the multiple-lines representation to be fully explicit. Figure \ref{f6} displays the various graphs that arise in this computations together with their associated algebraic counterpart. Our goal is to construct all graphs that can be built as in Figure \ref{f4} with $\ F=\frac{Z^{2}}{2!} \ $ and
$\ G=\frac{X^{2}}{2!}, \ $ proceeding in increasing order in the cardinality of $\ h$.

\begin{figure}[t]
\centering
\begin{tabular}{|p{4cm}|p{4cm}|p{4cm}|}\hline
\begin{center}
\includegraphics[width=4cm]{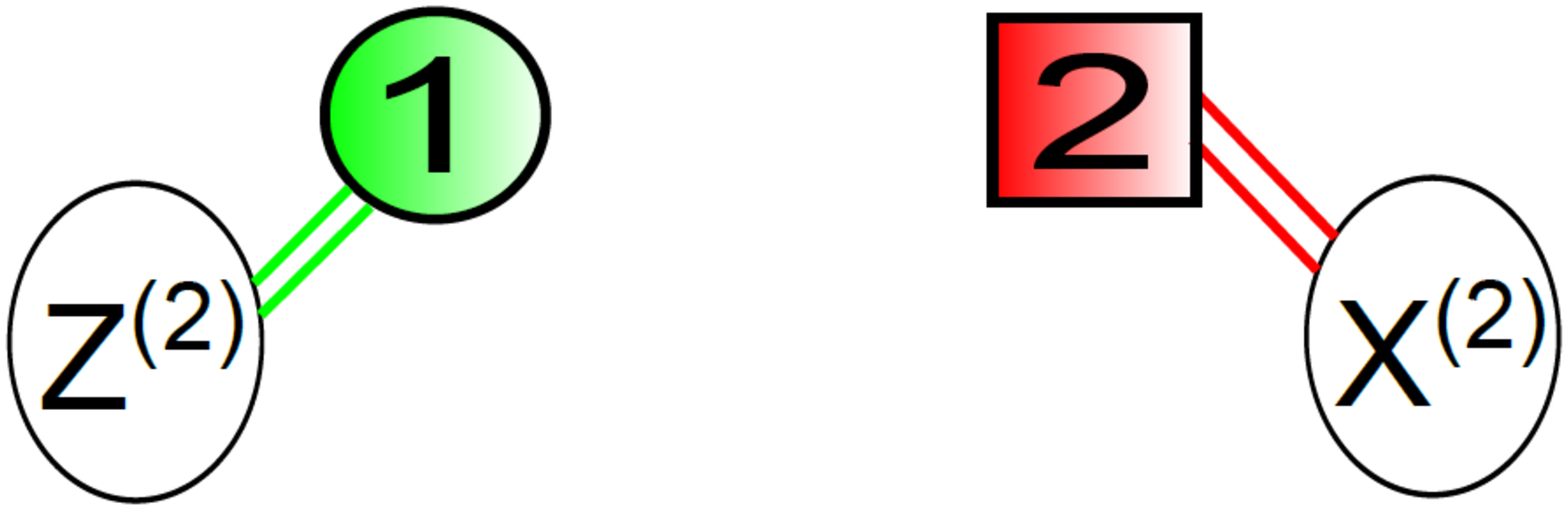}
\par\textbf{$\frac{x^{2}}{2!}\frac{z^{2}}{2!}$}
\end{center}
&
\begin{center}
\includegraphics[width=4cm]{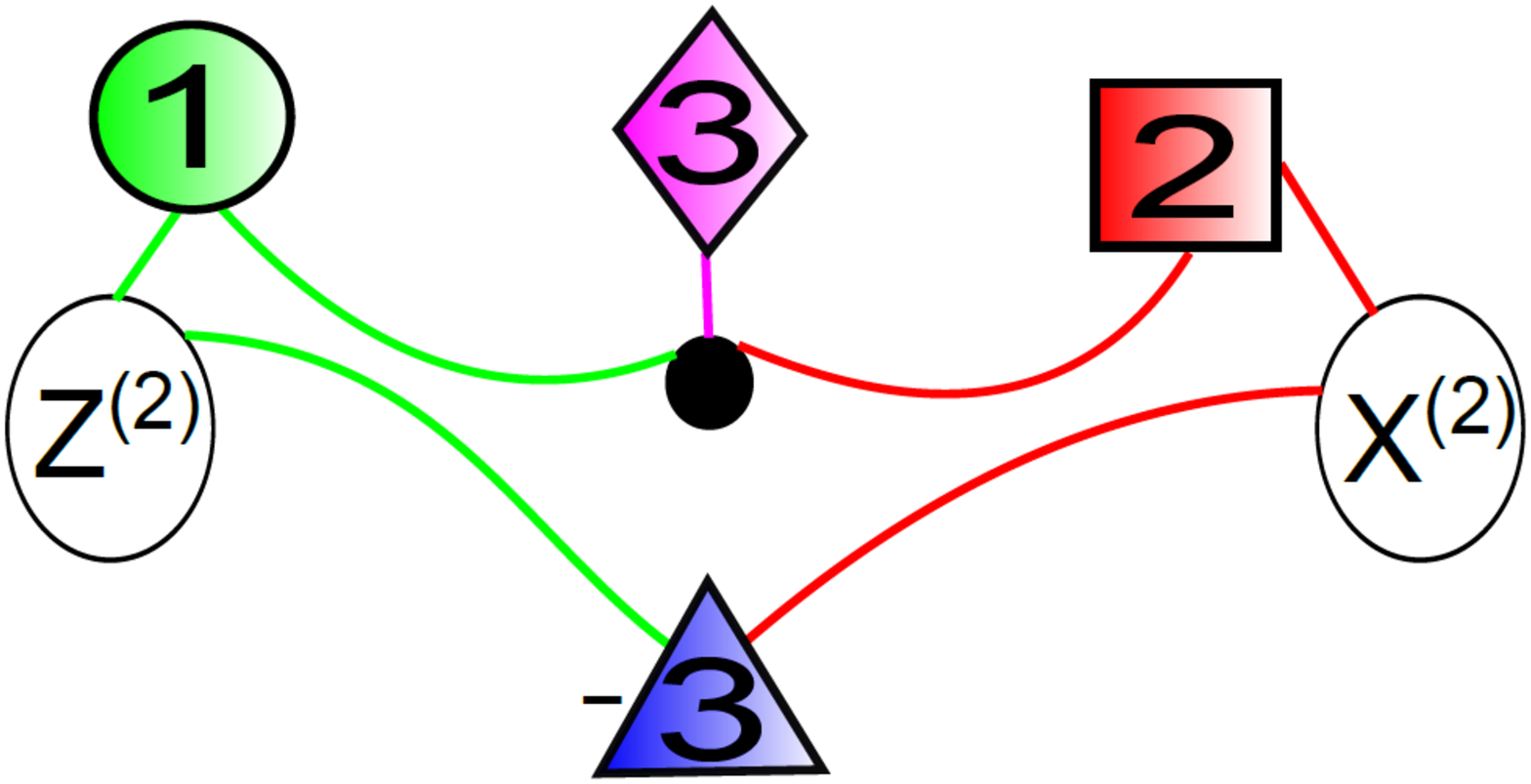}
\par\textbf{$-xyzh$}

\end{center}
&
\begin{center}
\includegraphics[width=4cm]{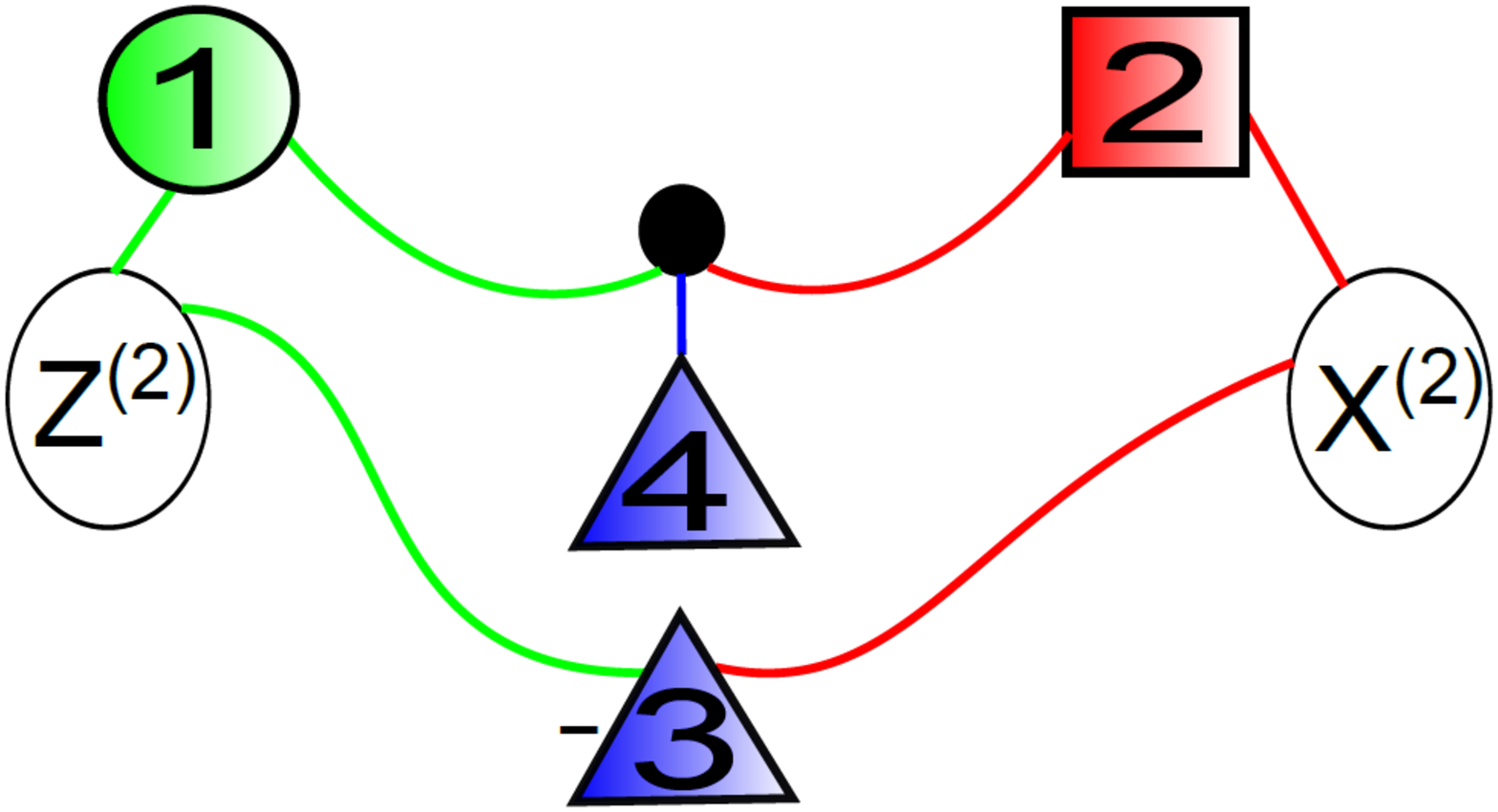}
\par\textbf{$-4xz\frac{h^{2}}{2!}$	}

\end{center}
\\\hline
\end{tabular}\\
\begin{tabular}{|p{4cm}|p{4cm}|p{4cm}|}\hline
\begin{center}
\includegraphics[width=4cm]{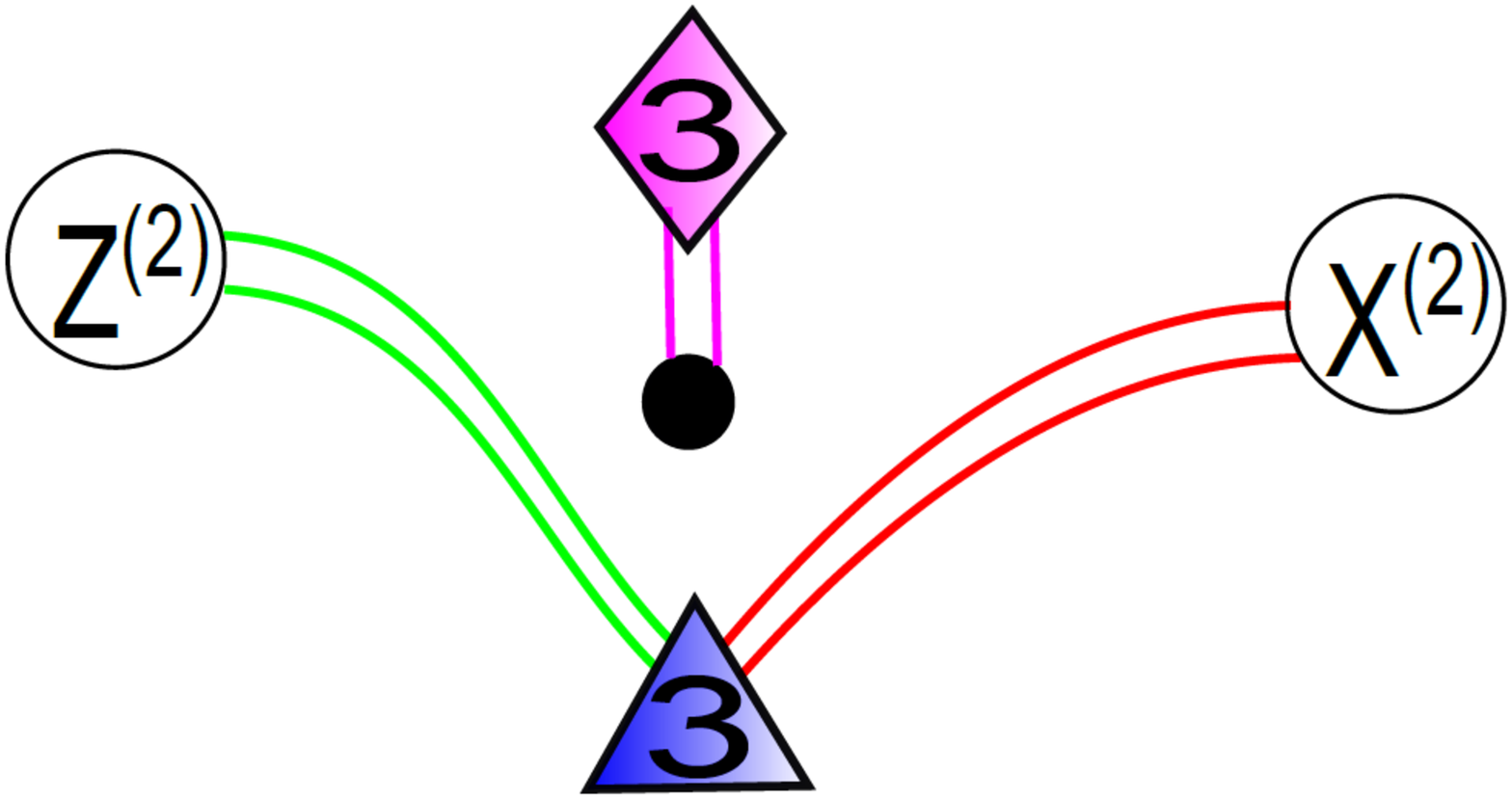}
\par\textbf{$2\frac{y^{2}}{2!}\frac{h^{2}}{2!}$}
\end{center}
&
\begin{center}
\includegraphics[width=4cm]{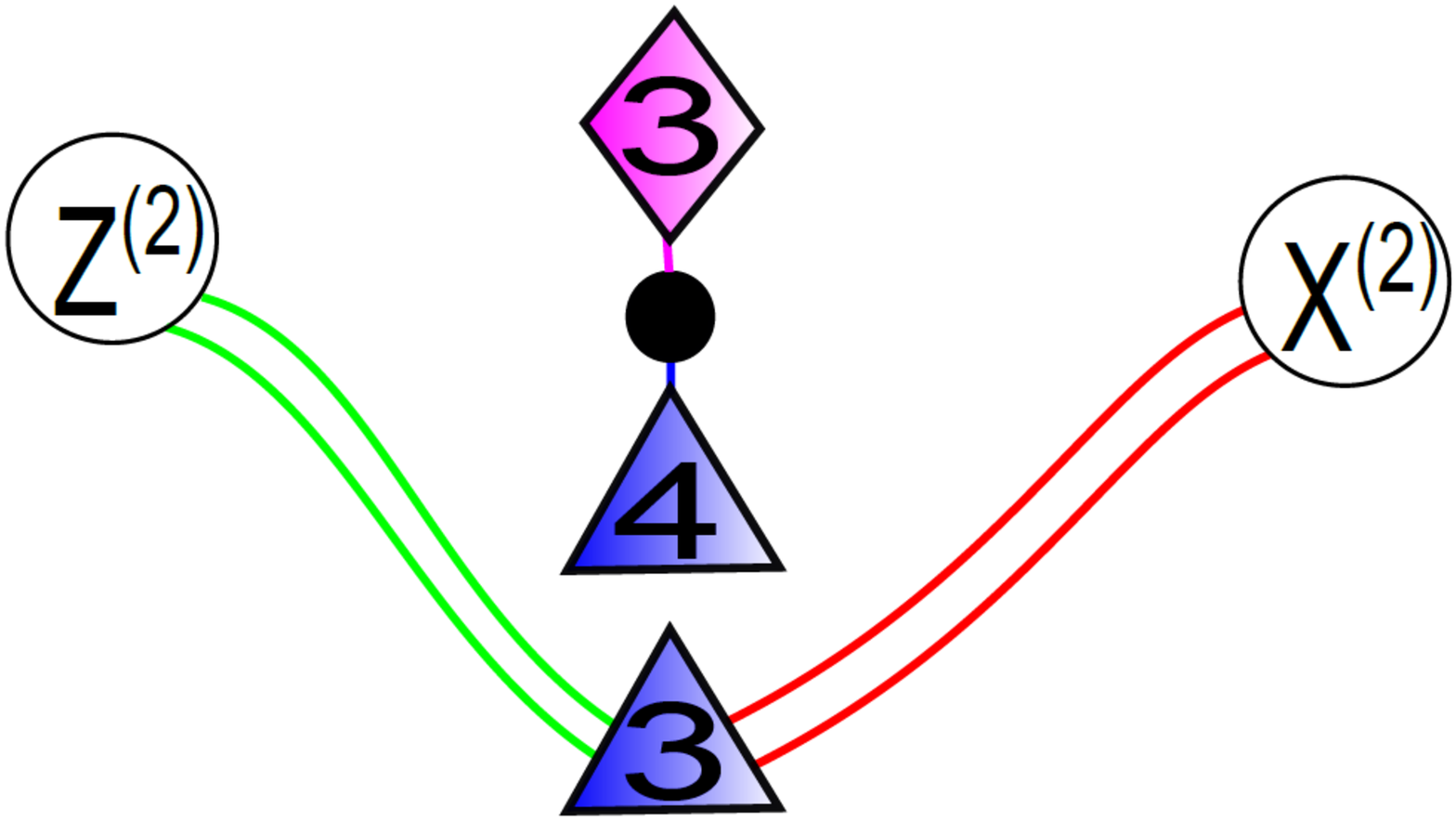}
\par\textbf{$3y\frac{h^{3}}{3!}$}

\end{center}

\\\hline
\end{tabular}\\

\caption{Graphical interpretation of the $\star$-product functor $\ \frac{Z^{2}}{2!}\star\frac{X^{2}}{2!}.$}\label{f6}
\end{figure}

\begin{prop}{\em The functor $\ \frac{Z^{2}}{2!}\star\frac{X^{2}}{2!} \in [ \mathbb{B}^4, C] \ $ is such that
$$\bigg|\frac{Z^{2}}{2!}\star\frac{X^{2}}{2!}(x,y,z,h)\bigg| \ = \  \left\{
\begin{array}{ccc}
1 & \mathrm{if} & |x|=|z|=1, \ y=h=\emptyset,\\
-1 & \mathrm{if}  & |x|=|y|=|z|=|h|=1,\\
-4 &  \mathrm{if} & |x|=|z|=1, \ |h|=2, \ y=\emptyset,\\
2 & \mathrm{if} & |y|=|h|=2, \ x=z=\emptyset,\\
3 & \mathrm{if} & |y|=1, \ |h|=3, \ x=z=\emptyset,\\
0 &  \mbox{otherwise}. &
\end{array}\right. $$ Therefore we have that
$$\frac{z^{2}}{2!}\frac{x^{2}}{2!}  \ = \ \bigg|\frac{Z^{2}}{2!}\frac{X^{2}}{2!}\bigg|  \ = \
\frac{x^{2}}{2!}\frac{z^{2}}{2!} \ - \  xyzh \ - \ 4xz\frac{h^{2}}{2!}
\ + \ 2\frac{y^{2}}{2!}\frac{h^{2}}{2!} \ + \  3y\frac{h^{3}}{3!} .$$

}
\end{prop}

\begin{proof}
The reader should have Figure \ref{f6} in mind as we develop our arguments.
\begin{itemize}
  \item The only graph we can build with  $\ |h|=0 \ $ is the one with two  $z$-lines attached to the left blob, and two $x$-lines attached to the right blob.
  \item There is only one graph with $\ |h|=1. \ $ Indeed, the unique $h$-colored element must necessarily lie in the block $\ h_3, \ $  given rise a $z$-colored edge connected to the left blob, as well as a $x$-colored edge connected to the right blob. Since $\ |h_3|=1,\ $ then either
      $\ |y_3|=1\ $ and $\ |h_4| =0, \ $ or $\ |y_3|=0\ $ and $\ |h_4|=1. \ $ The later option is not allowed since we are assuming that $\ |h|=|h_3|=1. \ $ Thus we have that $ \ (z_1+x_2)_{y_3+h_4}^{h_4} \ $ gives rise to a factor of  $\  (1+1)_{1}^{0} =1. \ $
  \item There are two cases with $\ |h|=2. \ $ Assume first that $\ |h_3|=|h_4|=1, \ $ then we obtain a factor of $\ 2 \ $ accounting for the partitions  $\ h \ $ in two blocks. Also we have that $ \ (z_1+x_2)_{y_3+h_4}^{h_4} \ $ gives rise to a factor of  $\ (1+1)_{1}^{1} =2. \ $ Thus we obtain the desired factor of $\ -4$.
  \item Next we assume that $\ |h|=|h_3|=2 \ $ which implies that $\ |h_4|=0 \ $ and $\ |y_3|=2. \ $ In this case $ \ (z_1+x_2)_{y_3+h_4}^{h_4} \ $ gives rise to a factor of  $\ (0+0)_{2+0}^{0} =  1.\ $
  \item Consider the case $\ |h|=3.\ $ We have that $\ |h_3|+|h_4|=3\ $ and $\ |y_3|+|h_4|=|h_3|.\ $ If $\ |h_3|=2\ $ and $\ |y_3|=|h_4|=1, \ $ then  $ \ (z_1+x_2)_{y_3+h_4}^{h_4} \ $ gives rise to a factor of  $\ (0+0)_{1+1}^{1} = (0)_2^1 = 0+1 = 1.\ $
  \item If $\ |h|\geq 5,\ $ then $\ |h_4| = |h|-|h_3| \geq 3 > |h_3|\ $ a contradiction since we know that $\ |h_4|\leq |h_3|.\ $ Thus there are no contribution to the product from such graphs.

  \item If $\ |h|=4,\ $ then we must have that $\ |h_3|=|h_4|=2\ $ and $\ |y_3|=0.\ $ Therefore  $ \ (z_1+x_2)_{y_3+h_4}^{h_4} \ $ gives rise to a factor of  $\ (0+0)_{0+2}^{2} =  0.\ $
\end{itemize}

\end{proof}

Next we consider the general case.

\begin{figure}[t]
\centering
\begin{tabular}{|p{8cm}|p{8cm}|p{8cm}|}\hline
\includegraphics[width=8cm]{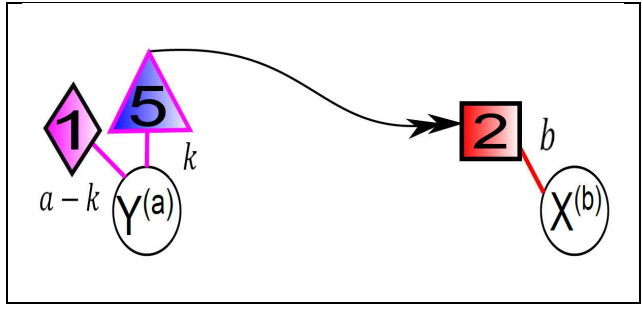}
\\\hline
\end{tabular}
\caption{Graphical interpretation of the $\star$-product functor $\ \frac{Y^{a}}{a!}\star \frac{X^{b}}{b!}.$}\label{f8}
\end{figure}

\begin{thm}\label{bb}
{\em Consider the category $\ ([ \mathbb{B}^4, C] ,\star). \ $ For $\ a,b \in \mathbb{N} \ $ we have that:

\begin{enumerate}
  \item  The functor $\ \frac{Y^{a}}{a!}\star \frac{X^{b}}{b!}\ $ is given by
$$\frac{Y^{a}}{a!}\star \frac{X^{b}}{b!}(x,y,z,h)  \ = \  \left\{
\begin{array}{cc}
[h,x \sqcup x] & \mathrm{if} \ |x|=b, \ |y|+|h|=a, \ z=\emptyset, \\
0 &  \mbox{otherwise}.
\end{array}\right. $$ Therefore we have that
$$ \frac{y^{a}}{a!}\star \frac{x^{b}}{b!} \ = \ \bigg|\frac{Y^{a}}{a!}\star \frac{X^{b}}{b!}\bigg| \ = \ \sum_{k=0}^{a} \; {(2b)^k }\frac{x^{b}}{b!}\frac{y^{a-k}}{(a-k)!}\frac{h^{k}}{k!}.$$

  \item The functor $\ \frac{Z^{a}}{a!}\star \frac{Y^{b}}{b!} \ $ is given by
$$\frac{Z^{a}}{a!}\star \frac{Y^{b}}{b!}(x,y,z,h) \ = \  \left\{
\begin{array}{cc}
[h,z \sqcup z] & \mathrm{if} \ x=\emptyset, \ |y|+|h|=b, \ |z|=a \\
0 &  \mbox{otherwise}.
\end{array}\right. $$ Therefore we have that
$$\frac{z^{a}}{a!}\star \frac{y^{b}}{b!} \ = \ \bigg|\frac{Z^{a}}{a!}\star \frac{Y^{b}}{b!}\bigg|\ = \
\sum_{k=0}^{b} \; (2a)^{k}\frac{y^{b-k}}{(b-k)!}\frac{z^{a}}{a!}\frac{h^{k}}{k!}.$$

\begin{figure}[t]
\centering
\begin{tabular}{|p{8cm}|p{8cm}|p{8cm}|}\hline
\includegraphics[width=8cm]{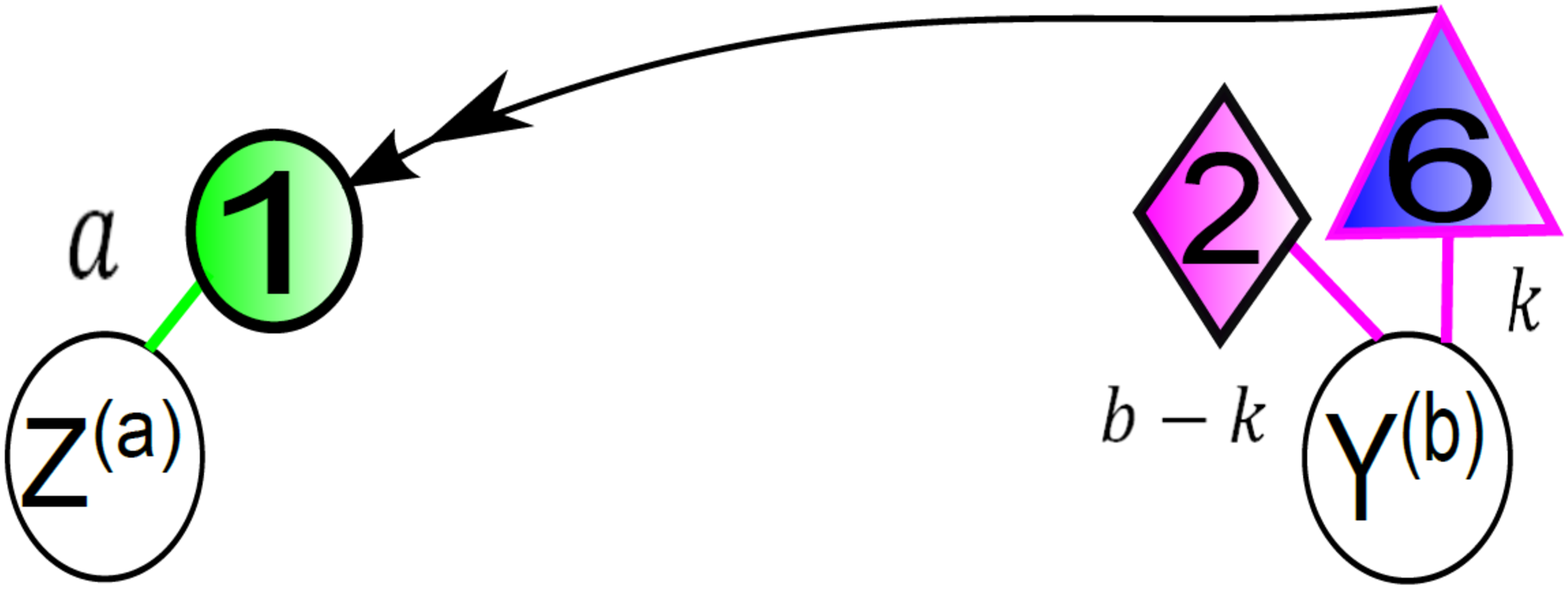}
\\\hline
\end{tabular}
\caption{Graphical interpretation of the $\star$-product functor $\ \frac{Z^{a}}{a!}\star \frac{Y^{b}}{b!}. $}\label{f7}
\end{figure}

\item The functor $\ \frac{Z^{a}}{a!}\star \frac{X^{b}}{b!} \ $ is such that
$$ \frac{Z^{a}}{a!}\star \frac{X^{b}}{b!} (x,y,z,h)  \ = \
\underset{w \sqcup v=h}{\bigoplus}(-1)^{|v|}\mathbb{L}(w)\otimes \mathbb{L}(y)\otimes (x \sqcup z)_{v}^{w} ,$$
where in the sum above the following conditions must be satisfied
$$|v| \leq \mathrm{min}(a,b), \ \ \ \ |x|+ |v|=b, \ \ \ \ |y|+ |w|=|v|,  \ \ \ \ |z|+|v|=a. $$

Therefore we have that $$ \frac{z^{a}}{a!}\star \frac{x^{b}}{b!} \ = \ \bigg|\frac{Z^{a}}{a!}\star \frac{X^{b}}{b!}\bigg|\ = $$
$$ \sum_{0\leq w \leq v \leq \mathrm{min}(a,b)} \;(-1)^{v}(v-w)!(v+w)_{w}(a+b-2v)^{w}_{v}\frac{x^{b-v}}{(b-v)!} \frac{y^{v-w}}{(v-w)!}\frac{z^{a-v}}{(a-v)!}\frac{h^{v+w}}{(v+w)!} .$$

\end{enumerate}

}
\end{thm}

\begin{proof} \ The reader should have Figures \ \ref{f8}, \ \ref{f7}, \ \ref{f9} \ in mind as we go along the proof. \\

\noindent 1. Clearly in this case the only non-empty block in the partition of $h$ is $h_5.$  Therefore
we have that $0\leq |h|=|h_5|\leq a.$ If $|h|=k,$ then we still need $a-k$ elements of color $y$, and $b$ elements of color $x$.
The map from $h_5$ to $[b] \sqcup [b]$ gives rise to the factor $(2b)^k.$ \\

\noindent 2.  Similarly in this case the only non-empty block in the partition of $h$ is $h_6.$  Therefore
we have that $0\leq |h|=|h_6|\leq b.$ If $|h|=k,$ then we still need $b-k$ elements of color $y$, and $a$ elements of color $x$.
The map from $h_6$ to $[a] \sqcup [a]$ gives rise to the factor $(2a)^k.$\\

\noindent 3. In this case $h$ is partitioned in two blocks $h_3$ and $h_4$ with $0 \leq |h_4| \leq |h_3| \leq \mathrm{min}(a,b).$
Set $w=|h_4|$ and $v=|h_3|.$ So we have that $|y_3|=v-w.$ Thus we need an additional set with $a-v$ elements with color $z$, and
another set with $b-v$ elements of colored $x$. Therefore $ \ (z_1+x_2)_{y_3+h_4}^{h_4}  \ $ gives rise to a factor of $\ (a-v + b-v)_{v-w+w}^{w} = (a+ b-2v)_{v}^{w}.\ $
\begin{figure}[t]
\centering
\begin{tabular}{|p{8cm}|p{8cm}|p{8cm}|}\hline
\includegraphics[width=8cm]{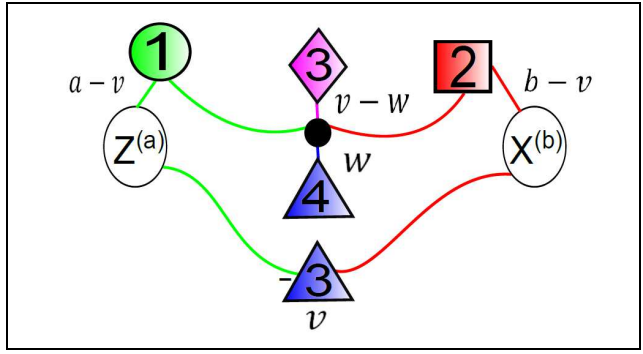}
\\\hline
\end{tabular}
\caption{ Graphical interpretation of the $\star$-product functor $\ \frac{Z^{a}}{a!}\star\frac{X^{b}}{b!}.$}\label{f9}
\end{figure}

\end{proof}

\subsection{Graphs and the  $\star$-product of the exponentiated variables}

In this final subsection we study with graphical methods the $\star$-product of the exponentiated variables in $\ \widehat{U}_{h}(\mathfrak{sl}_2). \ $ First we recall the combinatorial meaning of the exponentiated variables. The functor $E^X,$  similar constructions applied for the other variables, is given on a $4$-tuple of finite sets $\ (a,b,c,d) \in \mathbb{B}^4\ $ by
$$E^X(a,b,c,d) \ = \  \left\{
\begin{array}{cc}
1 & \mathrm{if} \ b=c=d=\emptyset,\\
0 & \  \mbox{otherwise}.
\end{array}\right. $$

We have that $\ |E^X|=e^x,\ $ and we  can similarly define functors $\ E^Y,\ E^Z \ $ and $\ E^H\ $ such that
$\ |E^Y|=e^y, \ |E^Z|=e^z\ $ and $\ |E^H|=e^h$.

\begin{thm}{\em Consider the category $\ ([ \mathbb{B}^4, C] ,\star).$ We have that:

\begin{enumerate}
  \item  The functor $\ E^Y\star E^X \ $ is such that $\ E^Y\star E^X(x,y,z,h) = \emptyset\ $ if $\ z \neq \emptyset, $ and otherwise it is given by
  $$E^Y\star E^X(x,y,z,h)  \ = \  [h,x \sqcup x],$$ and therefore
$$ e^y\star e^x \ = \ \big|E^Y\star E^X \big| \ = \ e^{xe^{2h}}e^y.$$

  \item  The functor $\ E^Z\star E^Y \ $ is such that $\ E^Z\star E^Y(x,y,z,h) = \emptyset\ $ if $\ y \neq \emptyset, \ $ and otherwise it is given by
  $$E^Z\star E^Y(x,y,z,h)  \ = \  [h,z \sqcup z],$$ and therefore
$$ e^z\star e^y \ = \ \big|E^Z\star E^Y \big| \ = \ e^xe^{ze^{2h}}.$$

\item The functor $\ E^Z\star E^X \ $ is given by  $$E^Z\star E^X(x,y,z,h) \ = \  \underset{w \sqcup v=h}{\bigoplus}(-1)^{|v|}\mathbb{L}(w)\otimes \mathbb{L}(y)\otimes (x \sqcup z)_{v}^{w}, $$     where in the sum above the identity $\ |y| + |w| =  |v| \ $ should hold.
    Therefore  $$ e^z\star e^x  \ = \ \big|E^Z\star E^X \big| \ = $$
$$ \sum_{a, c, w \leq v \in \mathbb{\mathbb{N}}} \ (-1)^{v}(v-w)!(v+w)_{w}(a+c)^{w}_{v}\ \frac{x^{a}}{a!} \frac{y^{v-w}}{(v-w)!}\frac{z^{c}}{c!}\frac{h^{v+w}}{(v+w)!} .$$

\end{enumerate}

}
\end{thm}

\begin{proof}
 The proof is similar to that of Theorem \ref{bb}. Again the reader should have Figures \ \ref{f8}, \ \ref{f7}, \ \ref{f9} \ in mind, but replacing the application of divided powers functors by the applications of  the corresponding exponentiated variables functors. Thus most of the  restrictions on the cardinality of sets are lifted. Item 3 follows then directly. Let us show item 1. From the previous considerations we have that:
$$e^y\star e^x \ = \ |E^Y\star E^X| \ = \ \sum_{a,b,c \in \mathbb{N}}(2a)^c\frac{x^a}{a!}\frac{y^b}{b!}\frac{h^c}{c!} \ = \
\sum_{a,c \in \mathbb{N}}\frac{x^a}{a!}\frac{(2ah)^c}{c!} e^y \ = $$
$$\sum_{a \in \mathbb{N}}\frac{x^a}{a!}e^{2ah} e^y  \ = \  \sum_{a \in \mathbb{N}}\frac{(xe^{2h})^a}{a!} e^y \ = \ e^{xe^{2h}}e^y.$$
\end{proof}

\section*{Acknowledgment}
E. Salamanca was partially supported by a "Young Researcher"  $-$  COLCIENCIAS grant.

\

\

\noindent ragadiaz@gmail.com \\
\noindent Departamento de Matem\'aticas, Pontificia Universidad Javeriana, Bogot\'a, Colombia\\

\

\noindent julisala14@gmail.com \\
\noindent FaMAF-CIEM (CONICET), Universidad Nacional de C\'ordoba, C\'ordoba, Argentina\\

\end{document}